\documentclass[12pt,twoside]{amsart}
\usepackage{latexsym,amsmath,amsopn,amssymb,amsthm,amsfonts}
\usepackage[T2A]{fontenc}
\usepackage[cp1251]{inputenc}
\usepackage[english]{babel}
\usepackage{graphicx}

\newtheorem{theorem}{Theorem}

\newtheorem{proposition}{Proposition}

\newtheorem{remark}{Remark}

\newtheorem{example}{Example}

\newtheorem{corollary}{Corollary}

\newtheorem{lemma}{Lemma}

\DeclareMathOperator{\ad}{ad}

\DeclareMathOperator{\SO}{SO}

{
\newcounter{table_counter}	
\newcommand{\Table}{\refstepcounter{table_counter}\newline\newline{\it Table \arabic{table_counter}}.}

\begin{document}

\vspace{0.5cm}
\title[Abnormal extremals of sub-Finsler  $q$-metrics on Lie groups]{Abnormal extremals of left-invariant sub-Finsler quasimetrics on four-dimensional Lie groups}
\author{V.~N.~Berestovskii, I.~A.~Zubareva}
\thanks{The first author is supported by the RFBR grant № 20-01-00661. 
The second author is supported by the program of fundamental scientific investigations of SD RAS № I.1.1., project № 0314--2019--0004.}
\address{Sobolev Institute of Mathematics,\newline
4 Koptyug Av., Novosibirsk, 630090, Russia}
\email{vberestov@inbox.ru}
\address{Sobolev Institute of Mathematics,\newline
13 Pevtsov Str., Omsk, 644099, Russia}
\email{i\_gribanova@mail.ru}
	
\begin{abstract}
Abnormal extremals on four-dimensional connected Lie groups with left-invariant sub-Finsler quasimetric, defined by a seminorm on a two-dimen\-si\-onal subspace of the Lie algebra generating the algebra, are found. 
In terms of structure constant of Lie algebra and supporting Minkowski function for the unit ball of seminorm on two-dimensional
subspace of Lie algebra, defining a quasimetric, we establish a criterion for strict abnormality of these extremals. 
 \vspace{2mm}
 
\noindent Mathematics Subject Classification (2010): 53C17, 53C22, 53C60, 49J15.
 \vspace{2mm}
 
\noindent {\it Keywords and phrases:} extremal, left-invariant sub-Finsler quasimetric, Lie algebra, 
 optimal control, polar, Pontryagin maximum principle,  (strictly) abnormal extremal, time-optimal problem.
\end{abstract}
\maketitle	

\section*{Introduction}

In \cite{Ber1}, the first author indicated that shortest arcs of a left-invariant (sub-)Finsler metric $d$ on any 
connected Lie group $G$, defined by a left--invariant bracktet generating distribution $D$ and a norm $F$ on 
$D(e)=\mathfrak{p}\subset TG_e=\mathfrak{g},$ are solutions of a left-invariant time-optimal problem with the 
closed unit ball $U$ of the normed vector space $(D(e),F)$ with zero center as the control region. The distribution $D$ 
is bracket generating if and only if the subspace $\mathfrak{p}$  generates the Lie algebra $(\mathfrak{g},[\cdot,\cdot])$ by Lie bracket $[\cdot,\cdot].$ Moreover, the statements about shortest arcs are also true for a pair  $(D(e),F)$ with a seminorm $F$, such that $F(u)>0$ for $0\neq u\in D(e),$ defining the left-invariant sub-Finsler quasimetric $d$ on $G$. 

The Pontryagin Maximum Principle (PMP) \cite{PBGM} gives some necessary conditions for solutions to the time-optimal problem. 
An extremal is any parameterized by the arc length curve in $G$, satisfying the Pontryagin Maximum Principle. 

An extremal can be normal or abnormal. Some extremals can be both normal and abnormal with respect to different controls (control functions); such extremals are called non-strictly abnormal. An abnormal extremal that is not non-strictly abnormal is called strictly abnormal.

In this paper, we solve the search problem for abnormal extremals on four-dimensinal connected Lie groups with a left-invariant sub-Finsler quasimetric, establish a criterion for the non-strict abnormality of these extremals, which allows us to formulate the criterion of their strict abnormality. Since the quasimetric is left-invariant we can assume that extremals initiate at the unit of the group. Every such abnormal extremal is some one-parametric subgroup of the Lie group.

A four-dimensional Lie group $G$ has abnormal extremals for all indicated seminorms $F$ on $\mathfrak{p}\subset\mathfrak{g}$
 if  $\dim(\mathfrak{p})=2$ and $\mathfrak{p}$ generates $\mathfrak{g}$. Such Lie groups $G$ and $\mathfrak{p}\subset\mathfrak{g}$ are considered (mainly) in this paper.

We were inspired to study these problems by memoir of W.~Liu and H.~Sussmann \cite{LZ} on strictly abnormal extremals of sub-Riemannian metrics, defined by bracket generating distributions of rank two. In particular, in \cite{LZ} they proved 
that any {\it regular} in their sense abnormal sub-Riemannian extremal is locally optimal, i.e. it is locally shortest arc, 
and presented an example of left-invariant metric on the Lie group $SO(3,\mathbb{R})\times \mathbb{R}$  with strictly abnormal extremal. Notice that all abnormal extremals of left-invariant sub-Finsler quasimetrics on four-dimensional connected Lie groups 
are regular in this sense.

Our solution to problems under consideration is naturally divided into three steps. At first, we need to select all four-dimensional real Lie algebras admitting two-dimensional subspaces generating them by Lie brackets. Then for any such Lie algebra it is necessary to find all pairwise nonequivalent classes of equivalent, i.e. transformable into each other by some automorphism of the Lie algebra, two-dimensional generating subspaces of this Lie algebra. At the end, for any seminorm on arbitrary representative of any such equivalence class we have to find abnormal extremals of the corresponding left-invariant sub-Finsler quasimetric on Lie group with a given Lie algebra and detect strictly abnormal among them, if they exist.

For all three steps we need a classification (up to isomorphism) of all four-dimensional real Lie algebras.
Such classification was obtained by G.M.~Mubarakzya\-nov in \cite{M}. We shall use in this paper a modified version of Mubarakzyanov's classification from paper \cite{BR} by R.~Biggs and C.~Remsing. In the paper, they indicate all automorphisms of 
Lie algebras under consideration and classified all connected Lie groups with these Lie algebras.

A.A.~Agrachev and D.~Barilari obtained in paper \cite{AB} an isometric classification of left-invariant sub-Riemannian metrics on three-dimensional Lie groups. Simultaneously they proved that, up to equivalence, all three-dimensional real Lie algebras, except 
$\mathfrak{sl}(2,\mathbb{R}),$ contain at most one generating two-dimensional subspace, and the algebra $\mathfrak{sl}(2,\mathbb{R})$ contains two such subspaces. We prove that the same result is true for one-dimensional extensions $\mathfrak{g}_{3}\oplus \mathbb{R}$ of these algebras for $\mathfrak{g}_3\neq\mathfrak{sl}(2,\mathbb{R}).$ Moreover, we  prove that, up to equivalence, the algebra
$\mathfrak{sl}(2,\mathbb{R})\oplus\mathbb{R}$ contains four generating two-dimensional subspaces, and every indecomposable Lie algebra $\mathfrak{g}_{4}$ contains at most one generating two-dimensional subspace. All other (main) results of the paper are obtained with application of results from \cite{Ber1}, \cite{PBGM},  \cite{BR}, \cite{BerZub}, \cite{H}.

\section{Preliminaries}

Recall the Engel theorem and the Cartan theorem, characterizing respectively {\it nilpotent} and {\it solvable} Lie algebras, \cite{J}, \cite{VGO}. 

A Lie algebra $(\mathfrak{g},[\cdot,\cdot])$ is nilpotent if and only if every linear operator
${\rm ad} X:= [X,\cdot]: \mathfrak{g}\rightarrow\mathfrak{g},$ $X\in\mathfrak{g},$ is nilpotent, i.e. $({\rm ad} X)^k=0$ for some natural number $k.$  A Lie algebra $(\mathfrak{g},[\cdot,\cdot])$ is solvable if and only if its derived ideal $\mathfrak{g}'=[\mathfrak{g},\mathfrak{g}]$ is nilpotent. 

A Lie algebra $\mathfrak{g}$ is {\it semisimple} (respectively, {\it simple}) if it has no nonzero solvable 
(respectively, improper) ideals (and $\mathfrak{g}'\neq 0$). Any simple Lie algebra is semisimple.

Lie algebras of interest to us have two types: $\mathfrak{g}_3\oplus\mathbb{R}$ and $\mathfrak{g}_4,$ where $\mathfrak{g}_3$ and $\mathfrak{g}_4$ are indecomposable three-dimensional and four-dimensional Lie algebras respectively.
Moreover, all  four-dimensional Lie algebras are solvable, except two ones:
$\mathfrak{g}_{3,7}\oplus\mathbb{R}=\mathfrak{so}(3,\mathbb{R})\oplus\mathbb{R}=\mathfrak{su}(2)\oplus\mathbb{R}$ and $\mathfrak{g}_{3,6}\oplus\mathbb{R}=\mathfrak{sl}(2,\mathbb{R})\oplus\mathbb{R};$ simply connected Lie groups with these Lie algebras
are the multiplicative group of nonzero quaternions and the universal covering of the group ${\rm GL}_e(2,\mathbb{R})$ 
respectively.

Table \ref{Tab:list} gives a classification of real Lie algebras of type $\mathfrak{g}_3\oplus\mathfrak{g}_1,$ where $\mathfrak{g}_1=\mathbb{R},$ and $\mathfrak{g}_4$ from \cite{BR}. For this, some basis $E_1,E_2,E_3,E_4$ is chosen in these algebras
and only nonzero brackets of this basis vectors are indicated. Let us give some comments about the Lie algebras from this Table.

Nilpotent algebras are only $\mathfrak{g}_{3,1}\oplus\mathfrak{g}_1=\mathfrak{h}\oplus\mathfrak{g}_1$
and $\mathfrak{g}_{4,1}=\mathfrak{eng},$ where $\mathfrak{h}=\mathfrak{nil}$ and $\mathfrak{eng}$ are Lie algebras of the Heisenberg group and the Engel group respectively. In addition, the Heisenberg and the Engel groups are so called {\it Carnot groups}. 
Maximal nilpotent ideals has dimensions 3 for the Lie algebra $\mathfrak{g}_{3,1}$ and 2 for all other solvable algebras of the type $\mathfrak{g}_3,$ and in the latter case it is commutative and equal to $\mathfrak{g}_3'=[\mathfrak{g}_3,\mathfrak{g}_3]$. 
Although the Heisenberg algebra has a two-dimensional generating subspace, it is easy to see that it is wrong for 
$\mathfrak{g}_{3,1}\oplus\mathfrak{g}_1.$ The Lie algebra $\mathfrak{g}_{3,3},$ and consequently,  $\mathfrak{g}_{3,3}\oplus\mathfrak{g}_1,$ do not have two-dimensional generating subspaces. All other Lie algebras of type $\mathfrak{g}_3\oplus\mathfrak{g}_1$ have them.

Maximal nilpotent ideals of Lie algebras of type $\mathfrak{g}_4\neq\mathfrak{g}_{4,1}$ have dimension 3 for  Lie algebras $\mathfrak{g}_{4,2}$ -- $\mathfrak{g}_{4,9}$ and  dimension 2 for the Lie algebra $\mathfrak{g}_{4,10}$; moreover,  maximal nilpotent ideals of Lie algebras $\mathfrak{g}_{4,2}$ -- $\mathfrak{g}_{4,6}$ are commutative, while maximal nilpotent ideals of Lie algebras  $\mathfrak{g}_{4,7}$ -- $\mathfrak{g}_{4,9}$ are isomorphic to the Heisenberg algebra \cite{M}. All Lie algebras $\mathfrak{g}_{4},$ with possible exception of some parameters in case of their parametric families, have two-dimensional generating subspaces.

Descriptions of Lie algebras $\mathfrak{g}_4,$ a choice of their bases, parameters in \cite{M} and \cite{BR}
coincide up to notation, except that for all families with parameters in \cite{BR} cases of extreme parameter values are highlighted,
while for the family  $\mathfrak{g}_{4,5}$ other parameter values are indicated. For Lie algebras $\mathfrak{g}_3$ it is true only for $\mathfrak{g}_{3,1}$ and $\mathfrak{g}_{3,7}.$ 

Every connected Lie group $G_{3,3}$ with Lie algebra $\mathfrak{g}_{3,3}$ and every connected Lie group $G$ with Lie algebra $\mathfrak{g}^{1,1}_{4,5}$ are simply connected and are characterized by the property that each their left-invariant Riemannian metric has constant negative sectional curvature, i.e. gives three-dimensional and four-dimensional Lobachevsky spaces \cite{Mil}.
Lie groups $G_{3,3}$ and $G$ are isomorphic to groups generated by homotheties (without rotations) and parallel translations of two-dimensional and three-dimensional Euclidean spaces, respectively.

\begin{table}[h]
\centering
\begin{tabular}{|p{4cm}|c|}
\hline
\centering{Type of Lie algebra}&Nonzero commutators\\
\hline
\centering{$\mathfrak{g}_{3,1}\oplus \mathfrak{g}_1$}&$[E_2,E_3] = E_1$\\
\hline
\centering{$\mathfrak{g}_{3,2}\oplus\mathfrak{g}_1$}&$ [E_2,E_3] = E_1-E_2,\quad [E_3,E_1] = E_1$\\
\hline
\centering{$\mathfrak{g}_{3,3}\oplus\mathfrak{g}_1$}&$[E_2,E_3] = -E_2,\quad [E_3,E_1] = E_1$\\
\hline
\centering{$\mathfrak{g}^0_{3,4}\oplus\mathfrak{g}_1$}&$[E_2,E_3] = E_1,\quad [E_3,E_1] = -E_2$\\
\hline
\centering{$\mathfrak{g}^{\alpha}_{3,4}\oplus\mathfrak{g}_1$,\quad\tiny{$0< \alpha\neq 1$}}&$[E_2,E_3]=E_1 -\alpha E_2,\quad [E_3,E_1] = \alpha E_1- E_2$\\
\hline
\centering{$\mathfrak{g}^{0}_{3,5}\oplus\mathfrak{g}_1$}&$[E_2,E_3] = E_1,\quad [E_3,E_1] = E_2$\\
\hline
\centering{$\mathfrak{g}^{\alpha}_{3,5}\oplus\mathfrak{g}_1$,\quad\tiny{$\alpha>0$}}&$[E_2,E_3] = E_1-\alpha E_2,\quad [E_3,E_1]=\alpha E_1+E_2$\\
\hline
\centering{$\mathfrak{g}_{3,6}\oplus\mathfrak{g}_1$}&$[E_2,E_3] = E_1,\quad [E_3,E_1] = E_2,\quad [E_1,E_2] = -E_3$\\
\hline
\centering{$\mathfrak{g}_{3,7}\oplus\mathfrak{g}_1$}&$[E_2,E_3] = E_1,\quad [E_3,E_1] = E_2,\quad [E_1,E_2] = E_3$\\
\hline
\centering{$\mathfrak{g}_{4,1}$}&$[E_2,E_4] = E_1,\quad [E_3,E_4] = E_2$\\
\hline
\centering{$\mathfrak{g}^{\alpha}_{4,2}$,\quad\tiny{$\alpha\neq 0$}}&$[E_1,E_4] =\alpha E_1,\quad [E_2,E_4] = E_2,\quad [E_3,E_4]=E_2+E_3$\\
\hline
\centering{$\mathfrak{g}_{4,3}$}&$[E_1,E_4] = E_1,\quad [E_3,E_4] = E_2$\\
\hline
\centering{$\mathfrak{g}_{4,4}$}&$[E_1,E_4] = E_1,\quad [E_2,E_4] = E_1+ E_2,\quad [E_3,E_4]=E_2+E_3$\\
\hline
\centering{$\mathfrak{g}^{\alpha,\beta}_{4,5}$}&$[E_1,E_4] = E_1,\quad [E_2,E_4] = \beta E_2,\quad [E_3,E_4]=\alpha E_3$\\
\centering{\tiny{$-1<\alpha\leq \beta\leq 1,\,\,\alpha\beta\neq 0$}}&$ $\\
\centering{\tiny{or $\alpha=-1,\,\,0<\beta\leq 1$}}&$ $\\
\hline
\centering{$\mathfrak{g}^{\alpha,\beta}_{4,6}$,\quad\tiny{$\alpha>0,\,\,\beta\in\mathbb{R}$}}&$[E_1,E_4] = \alpha E_1,\quad [E_2,E_4] = \beta E_2-E_3,\quad [E_3,E_4]=E_2+\beta E_3$\\
\hline
\centering{$\mathfrak{g}_{4,7}$}&$[E_1,E_4] = 2E_1,\quad [E_2,E_4] = E_2,$\\
\centering{$ $}&$[E_3,E_4]=E_2+E_3,\quad [E_2,E_3]=E_1$\\
\hline
\centering{$\mathfrak{g}^{-1}_{4,8}$}&$[E_2,E_3] = E_1,\quad [E_2,E_4] = E_2,\quad [E_3,E_4]=-E_3$\\
\hline
\centering{$\mathfrak{g}^{\alpha}_{4,8}$,\quad\tiny{$-1<\alpha\leq 1$}}&$[E_1,E_4] =(1+\alpha)E_1,\quad [E_2,E_4]=E_2,$\\
\centering{$ $}&$[E_3,E_4]=\alpha E_3,\quad [E_2,E_3]=E_1$\\
\hline
\centering{$\mathfrak{g}^{0}_{4,9}$}&$[E_2,E_3] = E_1,\quad [E_2,E_4] = -E_3,\quad [E_3,E_4]=E_2$\\
\hline
\centering{$\mathfrak{g}^{\alpha}_{4,9}$,\quad\tiny{$\alpha>0$}}&$[E_1,E_4] =2\alpha E_1,\quad [E_2,E_4]=\alpha E_2-E_3,$\\
\centering{$ $}&$[E_3,E_4]=E_2+\alpha E_3,\quad [E_2,E_3]=E_1$\\
\hline
\centering{$\mathfrak{g}_{4,10}$}&$[E_1,E_3] =E_1,\quad [E_2,E_3]=E_2,$\\
\centering{$ $}&$[E_1,E_4]=-E_2,\quad [E_2,E_4]=E_1$\\
\hline
\end{tabular}
\Table\label{Tab:list} Lie algebras $\mathfrak{g}_3\oplus\mathfrak{g}_1,$ $\mathfrak{g}_1=\mathbb{R},$ and $\mathfrak{g}_4$
\end{table}

\begin{table}[h]
\centering
\begin{tabular}{|p{4cm}|c|}
\hline
\centering{Type of Lie algebra}&Automorphisms\\
\hline
\centering{$\mathfrak{g}_{4,1}$}&\tiny{$\left(\begin{array}{cccc}
a_1a_2^2 & a_2a_3 & a_4 & a_5 \\
0 & a_1a_2 & a_3 & a_6 \\
0 & 0 & a_1 & a_7 \\
0 & 0 & 0 & a_2
\end{array}\right)$,\quad\tiny{$a_1a_2\neq 0$}}\\
\hline
\centering{$\mathfrak{g}^{\alpha}_{4,2}$,\,\,\tiny{$\alpha\neq 0$,\,\,$\alpha\neq 1$}}&\tiny{$\left(\begin{array}{cccc}
a_1 & 0 & 0 & a_4 \\
0 & a_2 & a_3 & a_5 \\
0 & 0 & a_2 & a_6 \\
0 & 0 & 0 & 1
\end{array}\right)$,\quad\tiny{$a_1a_2\neq 0$}}\\
\hline
\centering{$\mathfrak{g}_{4,3}$}&\tiny{$\left(\begin{array}{cccc}
a_1 & 0 & 0 & a_4 \\
0 & a_2 & a_3 & a_5 \\
0 & 0 & a_2 & a_6 \\
0 & 0 & 0 & 1
\end{array}\right)$,\quad\tiny{$a_1a_2\neq 0$}}\\
\hline
\centering{$\mathfrak{g}_{4,4}$}&\tiny{$\left(\begin{array}{cccc}
a_1 & a_2 & a_3 & a_4 \\
0 & a_1 & a_2 & a_5 \\
0 & 0 & a_1 & a_6 \\
0 & 0 & 0 & 1
\end{array}\right)$,\quad\tiny{$a_1\neq 0$}}\\
\hline
\centering{$\mathfrak{g}^{\alpha,\beta}_{4,5}$\linebreak\tiny{$-1<\alpha<\beta<1$,\,\,$\alpha\beta\neq 0$}\linebreak\tiny{or $\alpha=-1,\,\,0<\beta<1$}}&\tiny{$\left(\begin{array}{cccc}
a_1 & 0 & 0 & a_4 \\
0 & a_2 & 0 & a_5 \\
0 & 0 & a_3 & a_6 \\
0 & 0 & 0 & 1
\end{array}\right)$,\quad\tiny{$a_1a_2a_3\neq 0$}}\\
\hline
\centering{$\mathfrak{g}^{\alpha,\beta}_{4,6}$,\,\,\tiny{$\alpha>0$,\,\,$\beta\in\mathbb{R}$}}&\tiny{$\left(\begin{array}{cccc}
a_1 & 0 & 0 & a_4 \\
0 & a_2 & a_3 & a_5 \\
0 & -a_3 & a_2 & a_6 \\
0 & 0 & 0 & 1
\end{array}\right)$,\quad\tiny{$a_1\neq 0$,\quad $a_2^2+a_3^2\neq 0$}}\\
\hline
\centering{$\mathfrak{g}_{4,7}$}&\tiny{$\left(\begin{array}{cccc}
a_1^2 & -a_1a_3 & a_1a_4-(a_1+a_2)a_3 & a_5 \\
0 & a_1 & a_2 & a_4 \\
0 & 0 & a_1 & a_3 \\
0 & 0 & 0 & a_1
\end{array}\right)$,\quad\tiny{$a_1\neq 0$}}\\
\hline
\centering{$\mathfrak{g}^{-1}_{4,8}$}&\tiny{$\left(\begin{array}{cccc}
a_1a_2 & a_1a_3 & a_2a_4 & a_5 \\
0 & a_1 & 0 & a_4 \\
0 & 0 & a_2 & a_3 \\
0 & 0 & 0 & 1
\end{array}\right)$},\quad
\tiny{$\left(\begin{array}{cccc}
-a_1a_2 & -a_2a_4 & -a_1a_3 & a_5 \\
0 & 0 & a_1 & a_4 \\
0 & a_2 & 0 & a_3 \\
0 & 0 & 0 & -1
\end{array}\right)$,}\\
\centering{$ $}&\tiny{$a_1a_2\neq 0$}\\
\hline
\centering{$\mathfrak{g}^0_{4,8}$}&\tiny{$\left(\begin{array}{cccc}
a_1a_2 & a_3 & a_2a_4 & a_5 \\
0 & a_1 & 0 & a_4 \\
0 & 0 & a_2 & 0 \\
0 & 0 & 0 & 1
\end{array}\right)$,\quad\tiny{$a_1a_2\neq 0$}}\\
\hline
\centering{$\mathfrak{g}^{\alpha}_{4,8}$,\,\,\tiny{$-1<\alpha<1$,\,\,$\alpha\neq 0$}}&\tiny{$\left(\begin{array}{cccc}
a_1a_2 & -a_1a_3 & a_2a_4 & a_5 \\
0 & a_1 & 0 & a_4 \\
0 & 0 & a_2 & \alpha a_3 \\
0 & 0 & 0 & 1
\end{array}\right)$,\quad\tiny{$a_1a_2\neq 0$}}\\
\hline
\end{tabular}
\Table\label{Tab:auto} Automorphisms of Lie algebras $\mathfrak{g}_4$
\end{table}

\begin{table}[h]
\centering
\begin{tabular}{|p{4cm}|c|}
\hline
\centering{Type of Lie algebra}&Automorphisms\\
\hline
\centering{$\mathfrak{g}^0_{4,9}$}&\tiny{$\left(\begin{array}{cccc}
\sigma (a_1^2+a_2^2) &-\sigma a_1a_4+a_2a_5 & -a_1a_5-\sigma a_2a_4 & a_3 \\
0 & a_1 & a_2 & a_4 \\
0 & -\sigma a_2 & \sigma a_1 & a_5 \\
0 & 0 & 0 & \sigma
\end{array}\right)$,}\\
\centering{$ $}&\tiny{$\sigma=\pm 1,\quad a_1^2+a_2^2\neq 0$}\\
\hline
\centering{$\mathfrak{g}^{\alpha}_{4,9}$,\,\,\tiny{$\alpha>0$}}&\tiny{$\left(\begin{array}{cccc}
a_1^2+a_2^2 & \frac{-a_2(\alpha a_4-a_5)-a_1(a_4+\alpha a_5)}{1+\alpha^2} & \frac{a_1(\alpha a_4-a_5)-a_2(a_4+\alpha a_5)}{1+\alpha^2} & a_3 \\
0 & a_1 & a_2 & a_4 \\
0 & -a_2 & a_1 & a_5 \\
0 & 0 & 0 & 1
\end{array}\right)$,}\\
\centering{$ $}&\tiny{$a_1^2+a_2^2\neq 0$}\\
\hline
\centering{$\mathfrak{g}_{4,10}$}&\tiny{$\left(\begin{array}{cccc}
a_1 & \sigma a_2 & a_3 & \sigma a_4 \\
-a_2 & \sigma a_1 & a_4 & -\sigma a_3 \\
0 & 0 & 1 & 0 \\
0 & 0 & 0 & \sigma
\end{array}\right)$,\quad\tiny{$\sigma=\pm 1$,\quad $a_1^2+a_2^2\neq 0$}}\\
\hline
\end{tabular}
\Table\label{Tab:auto2} Automorphisms of Lie algebras $\mathfrak{g}_4$ (continuation)
\end{table}

\section{Algebraic results}
\label{alg} 
	
\begin{lemma} 
\label{base}	
Let  $(\mathfrak{g},[\cdot,\cdot])$ be a four-dimensional real Lie algebra, $\mathfrak{p}\subset \mathfrak{g}$ be a two-dimensional subspace, generating $\mathfrak{g}$ by Lie bracket $[\cdot,\cdot]$. Then there exists a basis $(e_1,e_2,e_3,e_4)$ in $\mathfrak{g}$ such that
$e_1, e_2\in \mathfrak{p}$ and 
\begin{equation}
\label{fghl}
[e_1,e_2]=e_3,\,\,[e_1,e_3]=e_4,\,\, C_{23}^4=0.
\end{equation}
For given vector $e_1\in\mathfrak{p}$ with this property there exist unique up to multiplication by nonzero real numbers
linearly independent vectors  $e_1,e_2,e_3,e_4$ in $\mathfrak{g}$ satisfying relations (\ref{fghl}). 
In addition, up to multiplication by nonzero real numbers, the vectors $e_2,$ $e_3$ do not depend on the choice of such vector $e_1.$
Moreover, if $[e_2,e_3]$ is not a linear combination of vectors $e_2$ and $e_3,$ then the vector $e_1$ can be chosen so that $C^2_{23}=0.$
\end{lemma}	
	
\begin{proof}
At first, let us choose any basis $(e_1,e_2)$ in $\mathfrak{p}.$ Then, due to the condition on $\mathfrak{p},$ vectors
$e_1, e_2, e_3:=[e_1,e_2]$ are linearly independent, moreover, vectors $e_1, e_2, e_3$ and $[e_1,e_3]$ or $[e_2,e_3]$ are
linearly independent. We can assume that $(e_1,e_2,e_3,e_4:=[e_1,e_3])$ is a basis of the Lie algebra $\mathfrak{g}.$ 	
If $C_{23}^4=0$ then relations (\ref{fghl}) are satisfied. Otherwise, replacing $e_2$ with vector $e_2':=e_2-C_{23}^4e_1$ 
and denoting $e_2'$ via $e_2$, we get the required basis in $\mathfrak{g}.$ 

The second statement follows from the construction of the basis $(e_1,e_2,e_3,e_4).$ 

It is clear that the third statement is true for the vector $e_3.$ 

For each such basis, let us take any vector of the form $e^{\prime}_1:=e_1+xe_2,$ $x\in\mathbb{R}.$ Then 
$[e_2,e_3]=\sum_{k=1}^3 a_ke_k$ for some numbers $a_k\in\mathbb{R},$ $k=1,2,3$, and
$$[e^{\prime}_1,e_2]=e_3,\,\,e^{\prime}_4=[e^{\prime}_1,e_3]=e_4+x[e_2,e_3],\,\, [e_2,e_3]=\sum_{k=1}^3 a_ke_k=a_1e^{\prime}_1+(a_2-xa_1)e_2+a_3e_3,$$
i.e. formulas (\ref{fghl}) are preserved under replacement $e_1$ with $e^{\prime}_1.$ 
In addition, due to (\ref{fghl}), a vector collinear to the vector $e_2$ does not have the above properties of the vector $e_1$. It completes the proof of the third statement. 

By definition, $[e_2,e_3]=C^1_{23}e_1+C^2_{23}e_2+C^3_{23}e_3.$ Therefore, if $[e_2,e_3]$ is not a linear combination of vectors $e_2$ and $e_3,$ then $C^1_{23}\neq 0.$ If $C^2_{23}\neq 0,$ then we replace $e_1$ with vector 
$e^{\prime}_1=e_1 + (C^2_{23}/C^1_{23})e_2$ and get
$$C^1_{23}e_1+C^2_{23}e_2 = C^1_{23}(e^{\prime}_1-(C^2_{23}/C^1_{23})e_2)+C^2_{23}e_2 = C^1_{23}e^{\prime}_1.$$
\end{proof}

\begin{proposition}
\label{const}	
For the basis $(e_1,e_2,e_3,e_4)$ in $\mathfrak{g}$ from Lemma \ref{base}, we have	
$$C_{24}^1= C_{24}^2=0,\quad C_{24}^3=C_{23}^2,\quad C_{24}^4=C_{23}^3.$$	
\end{proposition}

\begin{proof}
Due to (\ref{fghl}), the Jacobi identity 
$$[e_1,[e_2,e_3]]+[e_2,[e_3,e_1]]+[e_3,[e_1,e_2]]=0$$ 
is equivalent to the equality $[e_2,e_4]=[e_1,[e_2,e_3]],$ i.e.	
$$\sum\limits_{k=1}^4C_{24}^ke_k=\left[e_1,\sum\limits_{m=1}^4C_{23}^me_m\right]=C_{23}^2e_3+C_{23}^3e_4+C_{23}^4\sum\limits_{k=1}^4C_{14}^ke_k= C_{23}^2e_3+C_{23}^3e_4.$$
This gives us Proposition \ref{const}.  
\end{proof}	

\begin{proposition}
\label{solv3}
A real Lie algebra $\mathfrak{g}=\mathfrak{g}_3\oplus\mathfrak{g}_1$  with solvable indecomposable subalgebra $\mathfrak{g}_3$ has a two-dimensional generating subspace if and only if the operator ${\rm ad} E_3: \mathfrak{g}'\rightarrow\mathfrak{g}'$ is non-degenerate and has  non-eigenvectors, in other words, $\mathfrak{g}_3\neq \mathfrak{g}_{3,1}$, $\mathfrak{g}_3\neq\mathfrak{g}_{3,3}.$ 
In addition, two-dimensional subspace $\mathfrak{p}\subset \mathfrak{g}$ generates the algebra $\mathfrak{g}$
if and only if it has a basis of the form $(e_1,e_2=a+b),$ where $e_1\in \mathfrak{g}_3,$ $e_1\notin\mathfrak{g}',$ $a=\beta E_4,$ $\beta\neq 0,$ $b\in\mathfrak{g}'$ is a non-eigenvector vector of the operator ${\rm ad} E_3.$ Statements of Lemma \ref{base} are true for vectors $e_1,e_2,$ $[e_2,e_3]=0$ and vectors $e_1,e_2'=b$ generate the Lie algebra $\mathfrak{g}_3.$
\end{proposition}

\begin{proof}
It follows from Table \ref{Tab:list} that all these Lie algebras $\mathfrak{g}$ have central subalgebra $\mathfrak{g}_1,$ 
two-dimensional commutative ideal $\mathfrak{I}$, spanned by vectors $E_1,$ $E_2,$ and $\mathfrak{g}'=[E_3,\mathfrak{I}]\subset\mathfrak{I},$  moreover  $\mathfrak{g}'$ is one-dimensional
(and is a center of $\mathfrak{g}$) only in case $\mathfrak{g}_3=\mathfrak{g}_{3,1}$.

Assume that a subspace $\mathfrak{p}\subset\mathfrak{g}$ generates $\mathfrak{g}$ and $(e,f)$ is some basis for  
$\mathfrak{p}.$ Then in consequence of what has been said none of vectors $e,f$ can lie in $\mathfrak{I}$ or be collinear to $E_4.$
In addition, since any vector collinear to a basis vector can be added to onether basis vector, we can assume that
$e\in \mathfrak{g}_3,$ $e\notin\mathfrak{J},$ $f=a+b,$ where $a=\beta E_4,$ $\beta\neq 0,$ $b\in\mathfrak{I},$ moreover $b\neq 0.$
Furthermore, $b$ cannot be an eigenvector of the operator ${\rm ad} E_3:\mathfrak{J}\rightarrow\mathfrak{J}$ and this operator is non-degenerate; otherwise $\mathfrak{p}$ is a subalgebra or is contained in some three-dimensional subalgebra  $\mathfrak{f}\subset\mathfrak{g}$ and does not generate $\mathfrak{g}.$ 

As a consequence of the last statement, every two-dimensional subspace $\mathfrak{p}\subset\mathfrak{g}$ does not generate $\mathfrak{g}$ if  $\mathfrak{g}_3 =\mathfrak{g}_{3,3}$ or  $\mathfrak{g}_3=\mathfrak{g}_{3,1}.$

In all other cases under consideration, $\mathfrak{g}'=\mathfrak{I},$ the operator ${\rm ad}E_3: \mathfrak{g}'\rightarrow\mathfrak{g}'$ is non-degenerate and has non-eigenvectors. Let a subspace $\mathfrak{p}\subset \mathfrak{g}$ has a basis $(e_1=e,e_2=f)$ as in the statement of Proposition \ref{solv3}. Then $[e_1,e_2]=a_3[E_3,b]$ for some $a_3\neq 0$ and $[E_3,b]\in\mathfrak{g}'$ is also a non-eigenvector of the operator ${\rm ad} E_3$; otherwise $b=({\rm ad} E_3)^{-1}([E_3,b])$  would be an eigenvector.
Consequently, $e_3:=[e_1,e_2]\in\mathfrak{g}'$, $e_4:=[e_1,e_3]\in\mathfrak{g}'$, and vectors $e_3,e_4$ are not collinear. Then  vectors $e_1,\dots,e_4$ are linearly independent, i.e. $e_1,e_2$ generate the algebra $\mathfrak{g}$. In addition, $[e_2,e_3]=[b,e_3]=0,$ because $b,e_3\in\mathfrak{g}',$ so all statements of Lemma \ref{base} are true for vectors $e_1,e_2$ .

The very last statement of Proposition \ref{solv3} is obvious.	
\end{proof}

Proposition \ref{solv3} and paper \cite{AB} imply the following Corollary.

\begin{corollary}
Any two generating two-dimensional subspaces of Lie algebra $\mathfrak{g}$ from Proposition \ref{solv3} are transformed into each other by some automorphism of algebra  $\mathfrak{g}.$ 	
\end{corollary}

\begin{theorem}
\label{simple}
A real Lie algebra $\mathfrak{g}=\mathfrak{g}_3\oplus\mathfrak{g}_1$ with simple subalgebra $\mathfrak{g}_3$ has a two-dimensional generating subspace. In addition, a two-dimensional subspace $\mathfrak{p}\subset\mathfrak{g}$ generates the Lie algebra $\mathfrak{g}$ if and only if 1) $\mathfrak{p}_1=p(\mathfrak{p})\neq\mathfrak{p}$ for the projection $p:\mathfrak{g}\rightarrow\mathfrak{g}_3$; 2) the subspace $\mathfrak{p}_1$ is two-dimensional; 3) the restriction of the Killing form $k$ of the Lie algebra $\mathfrak{g}_3$ to $\mathfrak{p}_1$ is non-degenerate. For $\mathfrak{g}_{3,7}\oplus\mathfrak{g}_1$ (respectively, $\mathfrak{g}_{3,6}\oplus\mathfrak{g}_1$) there exists exactly one (respectively, exist four) equivalence class(es) of such subspaces $\mathfrak{p}\subset\mathfrak{g}.$ Moreover, there exist unique up to multiplication by nonzero real numbers basis vectors 
in $\mathfrak{p}$ such that either $e_1\in\mathfrak{g}_3,$ $e_2\notin\mathfrak{g}_3$ and $k(e_1,p(e_2))=0,$ or $e_1\notin\mathfrak{g}_3,$
$e_2\in \mathfrak{g}_3$ and both vectors $p(e_1),e_2$ are isotropic; in addition, vectors $e_1,e_2$ satisfy Lemma \ref{base} with  conditions $0\neq [e_2,e_3] || e_1,$ i.e. $C^1_{23}\neq 0,$ $C^2_{23}=C^3_{23}=0$ in the first case and $0\neq [e_2,e_3] || e_2,$ i.e. $C^2_{23}\neq 0,$ $C^1_{23}=C^3_{23}=0$ in the second case.
\end{theorem}

\begin{proof}
According to Table \ref{Tab:list}, 
$$\mathfrak{g}_3=\mathfrak{g}_{3,7}=\mathfrak{su}(2)=\mathfrak{so}(3)\quad\mbox{or}\quad \mathfrak{g}_3=\mathfrak{g}_{3,6}=\mathfrak{sl}(2,\mathbb{R}).$$
By the Cartan criterion, a Lie algebra $\mathfrak{g}$ is semisimple if and only if its {\it Killing form}
$$k_{\mathfrak{g}}(X,Y)=k(X,Y):={\rm tr}({\rm ad} X {\rm ad} Y),\quad X,Y\in\mathfrak{g},$$
is non-degenerate on $\mathfrak{g}.$ In addition, $k$ is invariant under the automorphism group ${\rm Aut}(\mathfrak{g})$ of 
 algebra  $\mathfrak{g},$ the connected component of identity ${\rm Aut}_e(\mathfrak{g})$ of the group ${\rm Aut}(\mathfrak{g})$ coincides with ${\rm Ad} G$ 
in case of connected Lie group $G$ with the Lie algebra $\mathfrak{g}$, hence the Lie algebra $L({\rm Aut}(\mathfrak{g}))$ is equal to ${\rm ad}\mathfrak{g}$ \cite{J}, \cite{VGO}. 

In our case, symmetric form $b:=-\frac{1}{2}k$ is more usual. Then in consequence of Table \ref{Tab:list}, vectors
$E_1,$ $E_2,$ $E_3$ are mutually orthogonal with respect to $b,$ moreover $b(E_i,E_i)=1,$ $i=1,2,3$, for $\mathfrak{g}_{3,7}$ and $b(E_1,E_1)=b(E_2,E_2)=-b(E_3,E_3)=1$ for $\mathfrak{g}_{3,6}.$ On the ground of what has been said, 
${\rm Aut}_e(\mathfrak{so}(3))={\rm SO}(3),$ ${\rm Aut}_e(\mathfrak{sl}(2,\mathbb{R}))={\rm SO}_0(2,1).$ The last Lie group is  so-called {\it restricted Lorentz group}, its two-sheeted covering is isomorphic to the Lie group ${\rm SL}(2,\mathbb{R})$ \cite{BerZub1}.
In addition, ${\rm Aut}(\mathfrak{so}(3))={\rm SO}(3),$
${\rm Aut}(\mathfrak{sl}(2,\mathbb{R}))={\rm SO}(2,1);$ the last group has two connected components, one of which is the group
${\rm SO}_0(2,1)$ (see Table 6 in \cite{BR}).  

Spheres $S^2(r)=\{X\in\mathfrak{so}(3):b(X,X)=r^2\},$ $r\geq 0,$ are action orbits of the group ${\rm SO}(3)$ on $\mathfrak{so}(3).$ 
As a corollary, ${\rm SO}(3)$ acts transitively on two-dimensional manifold of two-dimensional subspaces in $\mathfrak{so}(3)$ with  corresponding orientation, $b$-orthogonal to vectors  $X\in S^2(1),$ two-sheeted covering of the Grassmann manifold ${\rm Gr}(2,3)$ of two-dimensional subspaces. In addition, every such subspace generates  $\mathfrak{so}(3).$ In fact, by Table \ref{Tab:list}, 
the bracket $[X,Y]$ is the vector product for Euclidean space $(\mathfrak{so}(3),b)$; if $(X,Y)$ is an orthonormal basis of any two-dimensional subspace
$\mathfrak{p}\subset\mathfrak{so}(3),$ then $(X,Y,Z=[X,Y])$ is an orthonormal basis in $(\mathfrak{so}(3),b),$ moreover
\begin{equation}
\label{eucl}
[X,Y]=Z,\quad [Z,X]=Y,\quad [Y,Z]=X.
\end{equation}
In addition, the group ${\rm SO}(3)$ acts simply transitively on the manifold of ordered orthonormal pairs of vectors
$(X,Y)$ in $(\mathfrak{so}(3),b).$

The action orbits of the group ${\rm SO}_0(2,1)$ on $\mathfrak{sl}(2,\mathbb{R})$ are hyperboloids of one sheet
$S^2(r)=\{X:b(X,X)=r^2\},$ $r>0,$ consisting of {\it spacelike} vectors, the sheets in hyperboloids of two sheets
$S^2(-r)=\{X:b(X,X)=-r^2\},$ $r>0,$ consisting of {\it timelike} vectors, zero $\{0\},$
two connected components in $S^2(0)-\{0\}$ in the cone $S^2(0)=\{X: b(X,X)=0\}$ of {\it isotropic} vectors.

As a consequence, the Grassmann two-dimensional manifold ${\rm Gr}(2,3)$ of all two-dimensional subspaces $\mathfrak{p}\subset\mathfrak{sl}(2,\mathbb{R})$ splits into three ${\rm SO}_0(2,1)$-orbits: I) ${\rm Gr}^+(2,3),$ where the 
form $k$ is non-degenerate and sign-definite ($b$ is positive definite), every its element is $b$-orthogonal to exactly one vector from the connected component of $S^2(-1)$; II) ${\rm Gr}^-(2,3),$ where forms $k$ and $b$ are non-degenerate and 
non-sign-definite, every its element is $b$-orthogonal to exactly one vector from $S^2(1)$; III) ${\rm Gr}^0(2,3),$ where 
forms $k$ and $b$ are degenerate, its elements touch the cone $S^2(0)$. 
In addition, ${\rm Gr}^0(2,3)$ is a closed one-dimensional submanifold in ${\rm Gr}(2,3),$ dividing  ${\rm Gr}(2,3)$ into  open submanifolds ${\rm Gr}^+(2,3),$ ${\rm Gr}^-(2,3).$

For each element in an orbit of the form I)--III), we can choose  $b$-orthogonal basis of one among the types
I) $(X,Y)$, II) $(X,Z)$ or $(Z,X),$ III) $(X,W),$ where 
\begin{equation}
\label{hyp0}
b(X,X)=b(Y,Y)=-b(Z,Z)=1,\quad b(X,Y)=b(X,Z)=0;  
\end{equation}
\begin{equation}
\label{hypo}
b(X,W)=b(W,W)=0,\quad W\neq 0.  
\end{equation}
The group ${\rm SO}(2,1)$ acts simply transitively on manifolds of ordered vector pairs of any among four
mentioned types $(X,Y),$ $(X,Z),$ $(Z,X)$ or $(X,W)$ from (\ref{hyp0}), (\ref{hypo}).

We state that each element of the first two orbits generates $\mathfrak{sl}(2,\mathbb{R}),$ while each element from
${\rm Gr}^0(2,3)$ is a subalgebra in $\mathfrak{sl}(2,\mathbb{R}).$ In consequence of what has been said above, to prove
the last statement it is enough to consider two-dimensional subspaces with bases of a special form.
  
I) Let $(X=E_1,Y=E_2)$ be a $b$-orthogonal basis of two-dimensional subspace $\mathfrak{p}.$ Then according to Table
\ref{Tab:list}, 
\begin{equation}
\label{hyp}
Z:=[X,Y],\quad [Z,X]= -Y,\quad [Y,Z]= -X,\quad b(Y,Z)=0,    
\end{equation}
with fulfillment of all equalities from (\ref{hyp0}).

II) Defining $Z$ by the first formula from (\ref{hyp}) for vectors $X,Y$ in I), we get $Z=-E_3$ and basises of the form $(X,Z)$ or $(Z,X)$ in two-dimensional subspace $\mathfrak{p}$ with conditions from (\ref{hyp0}) and formulas (\ref{hyp}).

III) Let $(X=E_1,W=E_2+E_3)$ be a $b$-orthogonal basis of two-dimensional subspace $\mathfrak{p}.$ Then according to Table
\ref{Tab:list},
\begin{equation}
\label{hyp1}
Z:=[X,W]= - W,\quad b(X,X)=1,\quad b(W,W)=0.  
\end{equation}

Let us suppose that the subspace $\mathfrak{p}\subset\mathfrak{g}$ generates $\mathfrak{g}$. 

Let us prove that conditions 1)---3) are satisfied. 

If the condition 1) is not satisfied, then subalgebra generated by $\mathfrak{p}$ is contained in $\mathfrak{g}_3.$
If the condition 2) is not satisfied, then the subspace $\mathfrak{p}$ is a two-dimensional commutative subalgebra in $\mathfrak{g}.$ In case, when $\mathfrak{g}=\mathfrak{g}_{3,7}\oplus\mathfrak{g}_1$, the condition 3) is a consequence of the condition 2). If
$\mathfrak{g}=\mathfrak{g}_{3,6}\oplus\mathfrak{g}_1$ and the condition 3) is not satisfied, then due to conditions 1) and 2), 
as was shown, $\mathfrak{p}_1$ is a two-dimensional subalgebra in $\mathfrak{g}_3$ and therefore  $\mathfrak{p}$ generates a three-dimensional subalgebra in $\mathfrak{g}.$

Let all conditions 1)---3) be satisfied.

Then I) $\mathfrak{p}_1\in{\rm Gr}^+(2,3)$ or II) $\mathfrak{p}_1\in{\rm Gr}^-(2,3)$ for $\mathfrak{g}=\mathfrak{g}_{3,6}\oplus\mathfrak{g}_1.$

I) The space $\mathfrak{s}=\mathfrak{p}\cap\mathfrak{g}_3=\mathfrak{p}\cap\mathfrak{p}_1$ is one-dimensional and its nonzero 
vectors are $b$-spacelike. 
Let us choose any $b$-orthonormal vector $e_1\in\mathfrak{s}.$ It is easy to see that there exists a unique vector $e_2\in\mathfrak{p}$, up to multiplication by $\pm 1$, such that $(e_1,p(e_2)=Y)$ is a $b$-orthonormal basis in $\mathfrak{p}_1.$ It is clear that $e_2\notin\mathfrak{g}_3.$ 

II) Three cases are possible: a) nonzero vectors from $\mathfrak{s}$ are $b$-spacelike; b) nonzero vectors from $\mathfrak{s}$ $b$-timelike; с) the space $\mathfrak{s}$ consists of $b$-isotropic vectors.

Arguing similarly to case I), we obtain:

a) there exists a basis $(e_1,e_2)$ in $\mathfrak{p}$ such that $e_1=X\in\mathfrak{s},$ $e_2\notin\mathfrak{g}_3$ and $(X,Z=p(e_2))$ is a basis in
$\mathfrak{p}_1$ with conditions (\ref{hyp0});

b) there exists a basis $(e_1,e_2)$ in $\mathfrak{p}$ such that $e_1=Z\in\mathfrak{s},$ $e_2\notin\mathfrak{g}_3$ and $(Z,X=p(e_2))$
is a basis in $\mathfrak{p}_1$ with conditions (\ref{hyp0});

c) there exists a basis $(e_1,e_2)$ in $\mathfrak{p}$ such that $e_1\notin\mathfrak{g}_3,$ $e_2=W_2\in\mathfrak{s},$ moreover $(W_1=p(e_1),W_2=e_2)$ is a basis of isotropic vectors in $\mathfrak{p}_1,$ transformed by some element of the group $\xi\in SO(2,1)$ to the pair $(E_3+E_1,E_3-E_1)$.

Every element $\xi$ of the automorphism group of the Lie algebra $\mathfrak{g}$ transforms the subalgebra $\mathfrak{g}_{3}$ into itself isomorphically. Consequently, for $\mathfrak{g}=\mathfrak{g}_{3,6}\oplus\mathfrak{g}_1$ it saves two types I), II) of  subspaces $\mathfrak{p}_1$ and three types a), b), c) of subspaces $\mathfrak{s}\subset\mathfrak{p}_1$ for type II) and there exist at least four non-equivalent types of subspaces $\mathfrak{p}\subset\mathfrak{g}.$ 

Now, to complete the proof of Theorem \ref{simple}, it remains to show that all constructed basises $(e_1,e_2)$ 
have properties indicated in Theorem \ref{simple}, and such basises for different subspaces $\mathfrak{p}$
of one and the same among four indicated types are transformed into each other by some automorphism of the Lie algebra $\mathfrak{g}.$

I) Let $\mathfrak{g}_3=\mathfrak{g}_{3,7}.$ Then in consequence of (\ref{eucl}), 
$$e_3=[e_1,e_2]=[X,Y]=Z,\,\,e_4=[e_1,e_3]=[X,Z]=-Y,\,\,[e_2,e_3]=[Y,Z]=X=e_1$$
subject to equalities (\ref{hyp0}) with replacement $-b(Z,Z)$ by $b(Z,Z)$.

Let $\mathfrak{g}_3=\mathfrak{g}_{3,6}.$ Then in consequence of (\ref{hyp}), 
$$e_3=[e_1,e_2]=[X,Y]=Z,\,\,e_4=[e_1,e_3]=[X,Z]=-Y,\,\,[e_2,e_3]=[Y,Z]=-X=-e_1$$
with fulfillment of equalities (\ref{hyp0}).

II) a) In consequence of (\ref{hyp}), 
$$e_3=[e_1,e_2]=[X,Z]=-Y,\,\,e_4=[e_1,e_3]=[X,-Y]=-Z,\,\,[e_2,e_3]=[Z,-Y]=-X=-e_1$$
with fulfillment of (\ref{hyp0}).

b) In consequence of (\ref{hyp}),
$$e_3=[e_1,e_2]=[Z,X]=Y,\quad e_4=[e_1,e_3]=[Z,Y]=Y,\quad [e_2,e_3]=[X,Y]=Z=e_1$$
with fulfillment of equalities ({hyp0}).

c) $$e_3=[e_1,e_2]=[F(E_3+E_1),F(E_3-E_1)]=F(-2E_2),$$
$$e_4=F[E_3+E_1,-2E_2]=2F(E_3+E_1),$$ 
$$[e_2,e_3]=[F(E_3-E_1),-2F(E_2)]=-2F(E_3-E_1)=-2e_2.$$

It is clear from the above that for indicated bases $(e_1,e_2)$ of different subspaces $\mathfrak{p}$ of  
the same type among mentioned four, basises $(p(e_1),p(e_2))$ are translated into each other by some automorphism $\xi$ of the algebra $\mathfrak{g}_3.$ If at the same time $e_i\neq p(e_i)$ for $i=2$ or $i=1,$ then $e_i-p(e_i)\in \mathfrak{g}_1$ and $\xi(e_i)$ is defined by formula $\xi(e_i)=\xi(p(e_i))+\xi(e_i-p(e_i)),$ where $\xi(e_i-p(e_i))$ is a corresponding vector from $\mathfrak{g}_1.$
\end{proof}

\begin{proposition}
\label{solv41}
All real Lie algebras $\mathfrak{g}=\mathfrak{g}_4$ with three-dimensional commutative ideal $\mathfrak{I}$, except  $\mathfrak{g}^{1}_{4,2}$, $\mathfrak{g}^{\alpha,1}_{4,5}$,  $-1\leq\alpha\leq 1$, $\alpha\neq 0$, $\mathfrak{g}^{\alpha,\alpha}_{4,5}$, $-1<\alpha<1$, $\alpha\neq 0$, have two-dimensional generating subspaces.
A two-dimensional subspace $\mathfrak{p}\subset \mathfrak{g}$ generates $\mathfrak{g}$ if and only if it does not belong to $\mathfrak{I}$ and intersects by zero with any ideal $\mathfrak{J}\subset\mathfrak{I}$ of $\mathfrak{g}$, where $\mathfrak{J}\neq\mathfrak{I}$. Moreover, for any basis $(e_1,e_2)$ in $\mathfrak{p}$ with $e_2\in\mathfrak{p}\cap\mathfrak{I}$,  
statements of Lemma \ref{base} are true and $[e_2,e_3]=0.$ Any two two-dimensional generating subspaces of the Lie algebra $\mathfrak{g}$ are translated into each other by some its automorphism.
\end{proposition}

\begin{proof}
It follows from Table \ref{Tab:list} that three-dimensional commutative ideal $\mathfrak{I}$ of the algebra $\mathfrak{g}$ is spanned by vectors $E_1,$ $E_2$, $E_3$. Note that in case  $\mathfrak{g}=\mathfrak{g}^1_{4,2}$ any one-dimensional subspace with a basis vector of the form $xE_1+yE_2$, $x,y\in\mathbb{R}$, is a one-dimensional ideal of this algebra. In case $\mathfrak{g}=\mathfrak{g}^{\alpha,1}_{4,5}$, $-1\leq\alpha<1$, $\alpha\neq 0,$ (respectively, $\mathfrak{g}=\mathfrak{g}^{\alpha,\alpha}_{4,5}$, $-1<\alpha<1$, $\alpha\neq 0$) 
any two-dimensional subspace with basis vectors $E_3$, $xE_1+yE_2$ (respectively, $E_1$, $xE_2+yE_3$), where $x,y\in\mathbb{R}$, is a 
two-dimensional ideal of the algebra $\mathfrak{g}$, contained in $\mathfrak{I}$, and each of these basis vectors generates a one-dimensional ideal of the algebra $\mathfrak{g}$. Any one-dimensional (two-dimensional) subspace in  $\mathfrak{I}\subset\mathfrak{g}^{1,1}_{4,5}$ is a one-dimensional (two-dimensional) ideal of the algebra 
$\mathfrak{g}^{1,1}_{4,5}$.

Let us suppose that the subspace $\mathfrak{p}\subset\mathfrak{g}$ generates $\mathfrak{g}$ and $(e,f)$ is some basis for
$\mathfrak{p}.$ It is clear that at least one of vectors $e,\,f$ does not belong to $\mathfrak{I}$. Since any vector collinear to one basis vector can be added to onether vector of the basis, we can assume that $e\notin\mathfrak{I}$ and $f\in\mathfrak{I}$. 
Moreover, $f$ belongs to no ideal $\mathfrak{J}\subset\mathfrak{I}$ of the algebra $\mathfrak{g}$, different from $\mathfrak{I}$;
otherwise $\mathfrak{p}$ is a subalgebra or $\mathfrak{p}$ is contained in some three-dimensional subalgebra and thus does not generate $\mathfrak{g}.$ Therefore $\mathfrak{p}\cap\mathfrak{J}=\{0\}$.
From here and previous paragraph it follows that algebras $\mathfrak{g}_{4,2}^1$, $\mathfrak{g}^{\alpha,1}_{4,5}$,  $-1\leq\alpha\leq 1$, $\alpha\neq 0$, and $\mathfrak{g}^{\alpha,\alpha}_{4,5}$, $-1<\alpha<1$, $\alpha\neq 0$, have no two-dimensional generating subspace.

{\bf 1.} Let $\mathfrak{g}=\mathfrak{g}_{4,k}$, $k=1,4$.
At first, consider the case when the basis from Proposition \ref{solv41} has a special form $(e_1,e_2)=(E_4,E_3).$ 
Then Table \ref{Tab:list} implies that 
$$\mathfrak{g}_{4,1}:\quad e_3=[e_1,e_2]=-E_2,\quad e_4=[e_1,e_3]=E_1;$$
$$\mathfrak{g}_{4,4}:\quad e_3=[e_1,e_2]=-E_2-E_3,\quad e_4=[e_1,e_3]=E_1+2E_2+E_3.$$ 
It is easy to see that in all cases $(e_1,e_2,e_3,e_4)$ is a basis of the Lie algebra $\mathfrak{g}$ and $[e_2,e_3]=0,$  so vectors
$e_1,e_2$  satisfy Lemma \ref{base}.

Now it follows from Table \ref{Tab:auto} that for the basis $(e_1,e_2),$ specified in Proposition \ref{solv41}, there exists an 
automorphism $\xi$ of the Lie algebra $\mathfrak{g}$ such that $e_1=\xi(\beta E_4)$, $e_2=\xi(E_3)$, where $\beta\neq 0$.
Consequently, two-dimensional subspace $\mathfrak{p}\subset\mathfrak{g}$ with basis $(e_1,e_2)$ generates the Lie algebra  $\mathfrak{g}$, and for vectors $e_1$, $e_2$ all statements of Lemma \ref{base} are true, then all statements of Proposition \ref{solv41} are true.

{\bf 2.} Let $\mathfrak{g}=\mathfrak{g}_{4,3}$, $\mathfrak{g}=\mathfrak{g}^{\alpha}_{4,2}$,  $\alpha\notin\{0,1\}$, or 
$\mathfrak{g}=\mathfrak{g}^{\alpha}_{4,6}$, $\alpha>0$, $\beta\in\mathbb{R}$.
At first, consider the case when the basis from Proposition \ref{solv41} has a special form $(e_1,e_2)=(E_4,E_1+E_3).$ 
Then it follows from Table \ref{Tab:list} that 
$$\mathfrak{g}^{\alpha}_{4,2}:\quad e_3=[e_1,e_2]=-\alpha E_1-E_2-E_3,\quad e_4=[e_1,e_3]=\alpha^2 E_1+2E_2+E_3;$$
$$\mathfrak{g}_{4,3}:\quad e_3=[e_1,e_2]=-E_1-E_2,\quad e_4=[e_1,e_3]=E_1;$$ 
$$\mathfrak{g}^{\alpha}_{4,6}:\quad e_3=[e_1,e_2]=-\alpha E_1-E_2-\beta E_3,\quad e_4=[e_1,e_3]=\alpha^2 E_1+2\beta E_2+(\beta^2-1)E_3.$$ 
It is easy to see that in all cases $(e_1,e_2,e_3,e_4)$ is a basis of the Lie algebra $\mathfrak{g}$ and $[e_2,e_3]=0,$ so that  vectors $e_1,e_2$ satisfy Lemma \ref{base}.

Now it follows from Table \ref{Tab:auto} that for the basis $(e_1,e_2),$ specified in Proposition \ref{solv41}, there exists an 
automorphism $\xi$ of the Lie algebra  $\mathfrak{g}$ such that $e_1=\xi(\beta E_4)$, $e_2=\xi(E_1+E_3)$, where $\beta\neq 0$.  
Consequently, two-dimensional subspace $\mathfrak{p}\subset\mathfrak{g}$ with basis $(e_1,e_2)$ generates the Lie algebra  $\mathfrak{g}$, and for vectors $e_1$, $e_2$ all statements of Lemma \ref{base} are true, then all statements of Proposition \ref{solv41} are true.

{\bf 3.} Let $\mathfrak{g}=\mathfrak{g}^{\alpha,\beta}_{4,5}$, $-1<\alpha<\beta<1$, $\alpha\beta\neq 0$ or $\alpha=-1$, $0<\beta\leq 1$. At first, consider the case when the basis from Proposition \ref{solv41} has a special form $(e_1,e_2)=(E_4,E_1+E_2+E_3).$ 
Then it follows from Table \ref{Tab:list} that 
$$e_3=[e_1,e_2]=-E_1-\beta E_2-\alpha E_3,\,\,e_4=[e_1,e_3]=E_1+\beta^2E_2+\alpha^2E_3.$$
It is easy to see that $(e_1,e_2,e_3,e_4)$ is a basis of the Lie algebra $\mathfrak{g}$ and $[e_2,e_3]=0,$ so that vectors
$e_1,e_2$ satisfy Lemma \ref{base}.

Now it follows from Table \ref{Tab:auto} that for the basis $(e_1,e_2),$ specified in Proposition \ref{solv41}, there exists an 
automorphism $\xi$ of the Lie algebra $\mathfrak{g}$ such that $e_1=\xi(\beta E_4)$, $e_2=\xi(E_1+E_2+E_3)$, where $\beta\neq 0$.  
Consequently, two-dimensional subspace $\mathfrak{p}\subset\mathfrak{g}$ with the basis $(e_1,e_2)$ generates the Lie algebra   $\mathfrak{g}$, and for vectors $e_1$, $e_2$ all statements of Lemma \ref{base} are true, then all statements of Proposition \ref{solv41}
are true.  
\end{proof}	

\begin{theorem}
\label{solv49}
Real Lie algebras $\mathfrak{g}$, where $\mathfrak{g}=\mathfrak{g}_{4,7},$  $\mathfrak{g}=\mathfrak{g}^{\alpha}_{4,9}$, $\alpha\geq 0$, $\mathfrak{g}=\mathfrak{g}^{\alpha}_{4,8}$, $-1\leq\alpha\leq 1$, except $\mathfrak{g}^1_{4,8},$ have
two-dimensional generating subspaces. A two-dimensional subspace $\mathfrak{p}\subset \mathfrak{g}$ generates the algebra $\mathfrak{g}$ if and only if it is not contained in three-dimensional ideal $\mathfrak{I}$ and intersects by zero with any
ideal $\mathfrak{J}\subset\mathfrak{I},$ where $\mathfrak{J}\neq\mathfrak{I},$ of the algebra $\mathfrak{g}.$ Any two generating two-dimensional subspaces of the Lie algebra $\mathfrak{g}$ are translated to each other by an automorphism of this algebra.  
In addition, for the basis $(e_1,e_2)$ in $\mathfrak{p}$ with $e_1\in\mathfrak{p}\cap\mathfrak{I}$ are valid Lemma \ref{base}, 
moreover $C^1_{23}\neq 0$, $C^2_{23}=0$ for $\mathfrak{g}_{4,7},$ $\mathfrak{g}^{\alpha}_{4,9},$ 
$\mathfrak{g}^{\alpha}_{4,8},$ $\alpha\neq 0,$ and $C^1_{23}=C^2_{23}=0$ for $\mathfrak{g}^{0}_{4,8}.$   	 
\end{theorem}

\begin{proof}
It follows from Table \ref{Tab:list} that three-dimensional ideal $\mathfrak{I}$ of the Lie algebra $\mathfrak{g}$ is spanned by vectors $E_1$, $E_2$, $E_3$. In addition, the algebra  $\mathfrak{g}=\mathfrak{g}_{4,7}$ has unique (two-dimensional) ideal $\mathfrak{J}_1=\rm{span}(E_1,E_2)\subset\mathfrak{I}$, the algebra $\mathfrak{g}=\mathfrak{g}^{\alpha}_{4,9}$ has unique  (one-dimensional) ideal $\mathfrak{J}_2\subset\mathfrak{I}$ spanned on vector $E_1$, the algebra $\mathfrak{g}=\mathfrak{g}^{\alpha}_{4,8}$, $\alpha\neq 1$ has two two-dimensional commutative ideals $\mathfrak{J}_3=\rm{span}(E_1,E_3)\subset\mathfrak{I}$, $\mathfrak{J}_4=\rm{span}(E_2,E_3)\subset\mathfrak{I}$ and one
one-dimensional ideal $\mathfrak{J}_2\subset\mathfrak{I}$ spanned on vector $E_1$. Note also that every two-dimensional subspace with basis vectors $E_1$, $xE_2+yE_3$, where $x,y\in\mathbb{R}$, is a two-dimensional ideal of the algebra $\mathfrak{g}^1_{4,8}$.
	
Assume that a subspace $\mathfrak{p}\subset\mathfrak{g}$ generates the algebra $\mathfrak{g}$ and $(e,f)$ is a basis for 
$\mathfrak{p}.$ It is clear that at least one of vectors $e,\,f$ does not belong to $\mathfrak{I}$; otherwise the subspace $\mathfrak{p}\subset\mathfrak{I}$ does not generate $\mathfrak{g}$. Since to a vector of the basis one can add arbitrary 
vector wich is collinear to another one, then we can suppose that $e\in\mathfrak{I}$ and $f\notin\mathfrak{I}$, and also the
component of $f$ for $E_4$ in the basis $(E_1,E_2,E_3,E_4)$ is equal to 1. If $e$ belongs to some two-dimensional or one-dimensional subideal of $\mathfrak{I}$, then $\mathfrak{p}$ is a subalgebra or is contained in a three-dimensional subalgebra and therefore does not generate $\mathfrak{g}$. Then $\mathfrak{p}\cap\mathfrak{J}=\{0\}$. This and the previous paragraph imply that the algebra
$\mathfrak{g}^1_{4,8}$ has no two-dimensional generating subspace.
	
{\bf 1.} Assume that $\mathfrak{g}=\mathfrak{g}_{4,7}$. At first consider the special basis $(e_1,e_2)=(E_3,E_4).$ 
Then the Table \ref{Tab:list} implies that 
$$e_3=[e_1,e_2]=[E_3,E_4]=E_2+E_3,\quad e_4=[e_1,e_3]=-E_1,\quad [e_2,e_3]=e_1-2e_3.$$
It is easy to see that $(e_1,e_2,e_3,e_4)$ is a basis of the algebra $\mathfrak{g}$ and $C_{23}^1\neq 0$, 
$C_{23}^2=C^4_{23}=0,$ so the last statement of Theorem \ref{solv49} is true for the basis $(e_1,e_2)$. 
	
Since $e\notin\mathfrak{J}_2$ then the component $z$ of $e$ for $E_3$ is not zero. 
Let us set $e_1=e/z,e_2=f$. The conditions for $e_1,$ $e_2$ imply that they have a form 
$$e_1=xE_1+a_2E_2+E_3,\quad e_2=a_5E_1+a_4E_2+a_3E_3+E_4.$$ 
Let us show that for some change of the vector $e_2$ by $e_2^{\prime}=e_2+ye_1$, $y\in\mathbb{R}$ there exists an automorphism
$\xi$ of the algebra $\mathfrak{g}$ such that $\xi(E_3)=e_1,$ $\xi(E_4)=e^{\prime}_2$. According to Table \ref{Tab:auto}, 
the corresponding components for the vector $e_2^{\prime}$ are equal to $a_4^{\prime}=a_4+ya_2$, $a_3^{\prime}=a_3+y$ 
and so it must be
$$x=a^{\prime}_4-(1+a_2)a^{\prime}_3=a_4-(1+a_2)a_3 -y.$$ 
It is clear that there exists exactly one solution $y$ of this equation.
	
This, together with $[e_1,e_2^{\prime}]=[e_1,e_2]$, imply that the last statement of Theorem \ref{solv49} is valid for an appropriate basis $(e_1,e^{\prime}_2)$ of indicated view in the subspace $\mathfrak{p}$ . 
	
{\bf 2.} Assume that $\mathfrak{g}=\mathfrak{g}^{\alpha}_{4,9}$, $\alpha\geq 0$. 
	
Let us consider at first the case of special basis $(e_1,e_2)=(E_3,E_4)$ with properties from Theorem \ref{solv49}. Then it follows from Table \ref{Tab:list} that
$$e_3=[e_1,e_2]=E_2+ \alpha E_3,\quad e_4=[e_1,e_3]=-E_1,\quad [e_2,e_3]=(1+\alpha^2)e_1-2\alpha e_3.$$	
It is easy to see that $(e_1,e_2,e_3,e_4)$ is a basis of the Lie algebra $\mathfrak{g}$ and $C_{23}^1\neq 0$, $C_{23}^2=C^4_{23}=0,$ so the last statement of Theorem \ref{solv49} is true for vectors $e_1,e_2$. 
	
Set $e_1=e,e_2=f$. By conditions on vectors $e_1,e_2$, they have a form 
$$e_1=xE_1+a_2E_2+a_1E_3,\quad e_2=a_3E_1+a_4E_2+a_5E_3+E_4,\quad a_1^2+a_2^2\neq 0.$$
Let us show that under some change of $e_2$ by $e_2^{\prime}=e_2+ye_1$, $y\in\mathbb{R}$, there exists an automorphism $\xi$ 
of the algebra $\mathfrak{g}$ such that $\xi(E_3)=e_1,$ $\xi(E_4)=e^{\prime}_2$. In view of Table \ref{Tab:auto2}, 
the corresponding components of $e_2^{\prime}$ are equal to $a^{\prime}_5=a_5+ya_1,$ $a^{\prime}_4=a_4+ya_2.$ 

It must be
$$x=\frac{1}{1+\alpha^2}[a_1(\alpha a^{\prime}_4-a^{\prime}_5)-a_2(a^{\prime}_4+\alpha a^{\prime}_5)]=
\frac{1}{1+\alpha^2}[a_1(\alpha a_4-a_5)-a_2(a_4+\alpha a_5)-y(a_1^2+a_2^2)].$$
Because $a_1^2+a_2^2>0,$ there exists exactly one solution $y$ of this equation.
	
It follows from here and the equality $[e_1,e_2^{\prime}]=[e_1,e_2]$ that the last statement of Theorem \ref{solv49} is true
for an appropriate basis $(e_1,e^{\prime}_2)$ of the above form in $\mathfrak{p}.$ 	
	
{\bf 3.} Assume that $\mathfrak{g}=\mathfrak{g}^{\alpha}_{4,8}$, $-1\leq\alpha<1$. 
	
Let us consider at first the case when the basis from Theorem \ref{solv49} is of special view: $(e_1,e_2)=(E_2+E_3,E_4)$. 
Then in consequence of Table \ref{Tab:list}, 
$$e_3=[e_1,e_2]=E_2+\alpha E_3,\quad e_4=[e_1,e_3]=(\alpha-1)E_1,\quad [e_2,e_3]=-\alpha e_1-(1+\alpha)e_3.$$
Since $\alpha\neq 1$ then, as it is easy to see, $(e_1,e_2,e_3,e_4)$ is a basis of the Lie algebra $\mathfrak{g}$, $C_{23}^2=C_{23}^4=0$, moreover $C_{23}^1=0$ only for $\alpha=0$, thus the last statement of Theorem \ref{solv49} is valid for
vectors $e_1,e_2$. 
	
Set $e_1=e,e_2=f$. By conditions on the vector $e_1,$ it has a form $e_1=xE_1+a_1E_2+a_2E_3,$ where $a_1a_2\neq 0.$ 
We show that after some change of $e_2$ by $e_2^{\prime}=e_2+ye_1$, $y\in\mathbb{R}$ there exists an automorphism 
$\xi$ of $\mathfrak{g}$ such that $\xi(E_2+E_3)=e_1,$ $\xi(E_4)=e^{\prime}_2$.
	
Let $-1\leq\alpha<1$, $\alpha\neq 0$ and $e_2=a_5E_1+a_4E_2+\alpha a_3E_3+E_4.$ By Table \ref{Tab:auto}, the corresponding
components of $e_2^{\prime}$ are equal to $a_4^{\prime}=a_4+ya_1$, $a_3^{\prime}=a_3+ya_2/\alpha$ and 
$$x=a_2a_4^{\prime}-a_1a^{\prime}_3=a_2a_4-a_1a_3+a_1a_2(1-1/\alpha)y.$$
Since $a_1a_2\neq 0$, there exists exactly one solution $y$ of this equation.
	
Assume that $\alpha=0$ and $b_3$ is the component of the vector $e_2$ for $E_3.$
Since $a_2\neq 0$ then there exists a unique $y\in\mathbb{R}$ such that $ya_2+b_3=0.$ Then we can suppose that $e^{\prime}_2=e_2+ye_1$ has components as in the last column of the automorphism matrix in Table \ref{Tab:auto}. 
There exists a unique number $a_3$ such that $x=a_3+a_2a_4.$ Then $\xi(E_2+E_3)=e_1,$ $\xi(E_4)=e^{\prime}_2$ 
for an automorphism $\xi$ of the algebra $\mathfrak{g}$ with matrix from Table \ref{Tab:auto}.
Hence, since $[e_1,e_2^{\prime}]=[e_1,e_2]$, the last statement of Theorem \ref{solv49} is valid for an appropriate basis $(e_1,e^{\prime}_2)$ of indicated view in the subspace $\mathfrak{p}$. 
\end{proof}

\begin{proposition}
\label{solv410}
The real Lie algebra $\mathfrak{g}=\mathfrak{g}_{4,10}$ has two-dimensional geberating subspaces. A two-dimensional subspace  $\mathfrak{p}\subset \mathfrak{g}$ generates the algebra $\mathfrak{g}$ if and only if it is contained in no its three-dimensional 
ideal and intersects by zero with (unique) two-dimensional ideal of the algebra $\mathfrak{g}$. Any two generating two-dimensional subspaces of the Lie algebra $\mathfrak{g}$ can be translated to each other by an automorphism of this algebra. In addition, for a
basis $(e_1,e_2,e_3,e_4)$ from Lemma \ref{base}, $C^1_{23}=C^2_{23}=0.$ 	 
\end{proposition}

\begin{proof}
It follows from Table \ref{Tab:list} that the Lie algebra $\mathfrak{g}=\mathfrak{g}_{4,10}$ has two three-dimensional ideals
$\mathfrak{I}_1=\rm{span}(E_1,E_2,E_4)$ and $\mathfrak{I}_2=\rm{span}(E_1,E_2,E_3)$, moreover
$\mathfrak{J}=\mathfrak{I}_1\cap\mathfrak{I}_2=\rm{span}(E_1,E_2)$ is two-dimensional commutative ideal of the algebra
$\mathfrak{g}$.

Assume that the subspace $\mathfrak{p}\subset\mathfrak{g}$ generates the algebra $\mathfrak{g}$ and $(e,f)$ is a basis in $\mathfrak{p}.$ It is obvious that at least one from vectors $e,f$ does not belong to $\mathfrak{I}_1$; otherwise the 
subspace $\mathfrak{p}\subset\mathfrak{I}_1$ does not generate $\mathfrak{g}$. Since one vector of this basis could be sumed
with arbitrary vector, collinear with another vector of the basis, we can suppose that $f\notin\mathfrak{I}_1$ and $e\in\mathfrak{I}_1$. In addition, $e\notin\mathfrak{J}$; in the opposite case $\mathfrak{p}$ is a part of some three-dimensional subalgebra and so does not generate $\mathfrak{g}$. Without loss of generality we can suppose that the component of $e$ (respectively, $f$) for $E_4$ (respectively, $E_3$) in the basis $(E_1,E_2,E_3,E_4)$ is equal to 1. If $x$ is the component 
of $f$ for $E_4$ then $f^{\prime}=f-xe \in \mathfrak{I}_2$.

At first consider a special case $(e_1,e_2)=(E_4,E_1+E_3)$ of basises from Proposition \ref{solv410}. It follows from Table \ref{Tab:list} that
$$e_3=[e_1,e_2]=E_2,\quad e_4=[e_1,e_3]=-E_1,\quad [e_2,e_3]=-e_3.$$
Since $\alpha\neq 1$, then, as it is easy to see, $(e_1,e_2,e_3,e_4)$ is a basis of the Lie algebra $\mathfrak{g}$ and 
$C_{23}^4=0$, thus Lemma \ref{base} is true for vectors $e_1,e_2$, moreover $C_{23}^1=C_{23}^2=0$.

Set $e_1=e$, $e_2=f^{\prime}$. We can consider that the vector $e_1$ has components as in the last column of an automorphism
matrix in Table \ref{Tab:auto2} and $\sigma=1$. There exists a unique pair of numbers $(a_1,a_2)\in\mathbb{R}^2$ such that  components of $e_2$ are equal to the sum of corresponding components of the first and the third columns of an automorphism matrix in Table \ref{Tab:auto2}. In addition, $a_1^2+a_2^2\neq 0$; in the opposite case $[e_1,e_2]=0$. Therefore there exists an
automorphism $\xi$ of the algebra $\mathfrak{g}$ such that $\xi(E_4)=e_1$, $\xi(E_1+E_3)=e_2$.

Thus, since $[e_1,e_2^{\prime}]=[e_1,e_2],$ then for an appropriate basis $(e_1,e_2)$ of the subspace $\mathfrak{p},$ Lemma \ref{base}
is true, moreover $C^1_{23}=C^2_{23}=0.$
\end{proof}

\section{Main results}
\label{mains} 

In \cite{Ber1} is indicated that parameterized by arc length shortest curves of a left-invariant sub-Finsler metric $d$ on
any connected Lie group $G,$ defined by a left-invariant bracket generating distribution $D$ and a norm $F$ on $D(e),$ coincide 
with solutions to the time-optimal problem for the control system 
\begin{equation}
\label{a3}
\dot{g}(t)=dl_{g(t)}(u(t)),\quad u(t)\in U,
\end{equation}
with measureable control $u=u(t)$. Here $l_g(h)=gh$, the control region is the unit ball $U=\{u\in D(e)\,|F(u)\leq 1\},$
while the bracket generating of the distribution $D$ is equivalent to the condition that the corresponding subspace 
$\mathfrak{p}:= D(e)\subset \mathfrak{g}$
satisfies conditions of Lemma \ref{base}. It is clear that every parameterized by arc length shortest (curve) 
$g(t),\,\,0\leq t\leq a,$ satisfies (\ref{a3}) and $F(u(t))=1$ for almost all $t\in [0,a].$  

In addition, statements on shortest curves are true also for a pair $(D(e),F)$ with a seminorm $F$ such that $F(u)>0$ for 
$u\neq 0,$ defining a left-invariant sub-Finsler quasimetric $d$ on $G$. 

By the Pontryagin maximum principle (PMP) \cite{PBGM}, for time-optimality of a control $u(t)$ and corresponding trajectory 
$g(t),$  $t\in [0,a],$ it is necessary the existence of nowhere zero absolutely continuous covector-function 
$\psi(t)\in T^{\ast}_{g(t)}G$ such that for almost all $t\in [0,a]$ function $\mathcal{H}(g(t);\psi(t);u)=\psi(t)(dl_{g(t)}(u))$
of $u\in U$ attains its maximum at the point $u=u(t)$ 
\begin{equation}
\label{m}
M(t)=\psi(t)(dl_{g(t)}(u(t)))=\max\limits_{u\in U}\psi(t)(dl_{g(t)}(u))
\end{equation}
and an analogue of Hamilton-Jacobi system of ODE is fulfilled. In addition, the function $M(t),$  $t\in [0,a],$ is constant
and nonnegative, $M(t)\equiv M\geq 0$.

By {\it extremal} we shall mean a parameterized curve $g(t)$ in $G$ with maximally admitted connected domain 
$\Omega\subset\mathbb{R}$ which satisfies almost everywhere on maximal subset in $\Omega$ the Pontryagin maximum principle  
and conditions (\ref{a3}), $F(u(t))=1$ with a measureable function $u(t).$ In the case $M=0$ (respectively, $M>0$) an extremal is called {\it abnormal} (respectively, {\it normal}). In the normal case, proportionally changing $\psi=\psi(t)$, $t\in\mathbb{R}$, if it is necessary, one can assume that $M=1$.

\begin{proposition}
\label{prop1}
Every abnormal extremal in $(G,d)$ is one of two one-parameter subgroups 
\begin{equation}
\label{dif}
g(t)=\exp\left(\frac{ste_2}{F(se_2)}\right),\quad t\in \mathbb{R},\,\,s=\pm 1,
\end{equation}
or its left shift on $(G,d)$. 
\end{proposition}

\begin{proof}
We can consider (every value of) the covector function $\psi(t)\in T^{\ast}_{g(t)}G$ from PMP as a left-invariant $1$--form on $(G,\cdot)$ and so naturally identify it with a covector function $\psi(t)\in\mathfrak{g}^{\ast}=T^{\ast}_eG$.
	
In \cite{BerZub} for an extremal $g(t)\in G$ are proved the following relations, fulfilled for almost all $t$ in the domain: 
\begin{equation}
\label{difur}
\dot{g}(t)=dl_{g(t)}(u(t)),\,\, (\psi(t)(v))'=\psi(t)([u(t),v]),\,\,u(t), v\in \mathfrak{g},\quad F(u(t))=1. 
\end{equation}
	
Omitting for brevity the parameter $t,$ one can write the last equation in (\ref{difur}) as
$\psi'(v)=\psi([u,v]).$ In particular, for $\psi_i:=\psi(e_i),$ $i=1,2,3,4,$  
\begin{equation}
\label{ss}
\psi_i'= \psi([u,e_i]).
\end{equation} 
Let $u=u_1e_1+u_2e_2\in U$. We get from (\ref{ss}), (\ref{fghl}) and Proposition \ref{const}:
\begin{equation}
\label{psi1}
\psi_1'=\psi(-u_2e_3)=-u_2\psi_3,
\end{equation} 
\begin{equation}
\label{psi2}
\psi_2'=\psi(u_1e_3)=u_1\psi_3,
\end{equation}
\begin{equation}
\label{psi3}
\psi_3'=\psi(u_1e_4+u_2[e_2,e_3])=u_1\psi_4+u_2\sum\limits_{k=1}^3C_{23}^k\psi_k,
\end{equation}
\begin{equation}
\label{psi4}
\psi_4'=\psi(u_1[e_1,e_4]+u_2[e_2,e_4])=u_1\sum\limits_{k=1}^4C_{14}^k\psi_k+u_2\sum\limits_{k=3}^4C_{24}^k\psi_k.
\end{equation}
It is clear that in the abnormal case it must be $\psi_1=\psi_2\equiv 0.$ Then (\ref{psi1}) and (\ref{psi2}) 
imply that $\psi_3\equiv 0$. It follows from $\psi_4\neq 0$, (\ref{psi3}), (\ref{fghl}), and the equality $F(u)=1$ that 
\begin{equation}
\label{u12}
u_1\equiv 0,\quad u_2\equiv \frac{s}{F(se_2)},\quad s=\pm 1.
\end{equation}
It follows from \ref{u12}), Proposition \ref{const}, and (\ref{psi4}) that
\begin{equation}
\label{ ppsi4}
\psi_4(t)=\varphi_4\exp\left(\frac{C_{24}^4st}{F(se_2)}\right)=\varphi_4\exp\left(\frac{C_{23}^3st}{F(se_2)}\right),\,\,s=\pm 1,
\end{equation}
is a solution to the equation (\ref{psi4}) with the initial condition $\psi_4(0)=\varphi_4\neq 0$. It is obvious that it is
possible to find $u(t),$ $\psi(t)$ by the above formulae for all $t\in\mathbb{R}.$

Now Proposition \ref{prop1} is a direct corollary of (\ref{u12}) and the first equation in (\ref{difur}).
\end{proof}

With taking into account Proposition \ref{const} and equalities (\ref{u12}), every system (\ref{psi1})--(\ref{psi4}), fulfilled for almost all
$t\in\mathbb{R},$ has a form
\begin{equation}
\label{system}
\psi_1'=-u_2\psi_3,\,\,\psi_2'=0,\,\,\psi_3'=u_2\sum\limits_{k=1}^3C_{23}^k\psi_k,\,\,
\psi_4'=u_2\left(C_{23}^2\psi_3+C_{23}^3\psi_4\right),\,\,u_2\equiv \frac{s}{F(se_2)}.
\end{equation}
Then equations (\ref{system}) are fulfilled for all $t\in\mathbb{R}$ and (\ref{system}) is a system of linear ordinary differential equations with constant coefficients, so all its solutions are real-analytic. 

Below  $F(u_1,u_2):=F(u),$ $F_U$ is the supporting Minkowski function of $U$:
\begin{equation}
\label{FU}
F_U(x,y)=\max\limits_{(u_1,u_2)\in U}\left(xu_1+yu_2\right).
\end{equation}

The following theorem is valid.

\begin{theorem}
\label{main1}
An abnormal extremal (\ref{dif}) is non-strictly abnormal if and only if for the basis $(e_1,e_2,e_3,e_4)$
from Lemma \ref{base}, either $C_{23}^1=C_{23}^2=0$, or $C_{23}^1\neq 0$ and $F_U\left(0,s\right)=1/F(0,s),$ $s=\pm 1$.
\end{theorem}

\begin{proof} 
{\sl Necessity.} Assume that an abnormal extremal (\ref{dif}) is non-strictly abnormal. Then according to what has been said above 
there exists nowhere equal to zero real-analytic covector-function $\psi(t),$ which is a solution of ODE system (\ref{system}), and $\psi(t)(u(t))=F_U(\psi_1(t),\psi_2(t))\equiv 1$ for almost all $t\in\mathbb{R}.$ This and (\ref{u12}) imply that 
\begin{equation}
\label{fp}
\psi_2(t)\equiv 1/u_2,\quad F_U\left(\psi_1(t),1/u_2\right)\equiv 1,	
\end{equation} 
and points $\left(\psi_1(t),1/u_2\right)$, $t\in\mathbb{R}$, are dual for point $(0,u_2)$. Therefore the range of function $\psi_1(t)$, $t\in\mathbb{R}$, is a segment (degenerating to a point if the function $F$ is differentiable at the point $(0,u_2)$) because the body  $U^{\ast},$ dual to $U,$ is convex and bounded.
	
The system (\ref{system}) implies the ordinary differential equation
\begin{equation}
\label{equat}
\psi_1^{\prime\prime}-u_2C_{23}^3\psi_1^{\prime}+u_2^2C_{23}^1\psi_1+u_2C_{23}^2=0.
\end{equation} 
	
Assume that $C_{23}^1\neq 0$. On the ground of Lemma \ref{base}, we can suppose that $C_{23}^2=0$.  
Let us set $B:=\left(C_{23}^3\right)^2-4C_{23}^1$. Taking into account the theory from \cite{H}, it is easy to see that 
general solution of the equation (\ref{equat}) has a form 
$$\psi_1(t)=\left\{\begin{array}{lr}
A_1e^{\lambda_1t}+A_2e^{\lambda_2t},\,\,\lambda_{1,2}=u_2\left(C_{23}^3\pm\sqrt{B}\right)/2,\quad\text{если } B>0, \\
(A_1t+A_2)e^{\frac{1}{2}C_{23}^3u_2 t},\quad\text{если }B=0, \\
e^{\frac{1}{2}C_{23}^3u_2t}\left(A_1\cos\frac{u_2\sqrt{-B}t}{2}+A_2\sin\frac{u_2\sqrt{-B}t}{2}\right),\quad\text{если }B<0,
\end{array}\right.$$
where $A_1,\,A_2$ are arbitrary real numbers.
Since the function $\psi_1(t)$, $t\in\mathbb{R}$, is bounded then either $\psi_1(t)\equiv 0$ or
\begin{equation} 
\label{equat1}
\psi_1(t)=A_1\cos\frac{u_2\sqrt{-B}t}{2}+A_2\sin\frac{u_2\sqrt{-B}t}{2},\quad C_{23}^1<0,\,\,C_{23}^3=0.
\end{equation} 
Notice that for any $A_1,\,A_2\in\mathbb{R}$ the range of the function $\psi_1(t)$, defined by formula (\ref{equat1}), 
includes  $0$. Therefore in consequence of (\ref{fp}), $F_U\left(0,1/u_2\right)=1$ which is equivalent to the corresponding
equality in Theorem \ref{main1}.
	
Assume that $C_{23}^1=0$. By the theory from \cite{H}, general solution of the equation (\ref{equat}) has a form
$$\psi_1(t)=\left\{\begin{array}{lr}
A_1e^{C_{23}^3u_2t}+C_{23}^2t/C_{23}^3+A_2,\quad\text{если } C_{23}^3\neq 0, \\
-\frac{1}{2}C_{23}^2u_2t^2+A_1t+A_2,\quad\text{если }C_{23}^3=0, 
\end{array}\right.$$
where $A_1,\,A_2$ are arbitrary real numbers. Since the function $\psi_1(t)$, $t\in\mathbb{R},$ is bounded this implies that 
$C_{23}^2=0$ and $A_1=0$.
	
{\sl Sufficiency.}
Since $U$ is a convex bounded set in $\mathbb{R}^2$, $0\in{\rm Int}(U)$, and $F(0,u_2)=1$, where $u_2$ is defined in (\ref{u12})
for $s=1$ or $s=-1$, then there exists at least one $k\in\mathbb{R}$ such that $F_U(k,1/u_2)=1$. 
	
If $C_{23}^1=C_{23}^2=0$, we set
\begin{equation}
\label{funct}
\psi_1(t)\equiv k,\quad \psi_2(t)\equiv 1/u_2,\quad\psi_3(t)\equiv 0,\quad\psi_4(t)=e^{C_{23}^3u_2t}.
\end{equation} 
	
It is easy to see that functions $\psi_i(t)$, $i=1,\dots,4$, satisfy the ODE system (\ref{system}), 
\begin{equation}
\label{cond}
F_U(\psi_1(t),\psi_2(t))\equiv 1.
\end{equation} 
Consequently, abnormal extremal (\ref{dif})  satisfies PMP with $M(t)\equiv 1$ (see (\ref{m})), so it is not strictly normal.
	
If $C_{23}^1\neq 0$ then according to Lemma \ref{base} we can suppose that $C_{23}^2=0$. Let us define functions $\psi_i(t)$, $i=1,\dots,4$, by formulae (\ref{funct}), setting $k=0$.
It is easy to see that functions $\psi_i(t)$, $i=1,\dots,4$, satisfy the ODE system (\ref{system}). Furthermore, conditions
$F_U\left(0,s\right)=1/F(0,s),$ $s=\pm 1,$ and (\ref{u12}) imply (\ref{cond}). Thus abnormal extremal (\ref{dif}) 
satisfies PMP with $M(t)\equiv 1$ (see (\ref{m})) and therefore is not strictly abnormal.
	
Theorem \ref{main1} is proved.
\end{proof}

\begin{remark}
\label{subr}
If $d$ is a left-invariant sub-Riemannian metric on a connected 4-dimen\-sional Lie group $G$ with Lie algebra $\mathfrak{g},$ defined by an inner product $(\cdot,\cdot)$ on two-dimensional subspace $\mathfrak{p}\subset\mathfrak{g},$ $C^1_{23}\neq 0,$ and 
$(e_1,e_2,e_3,e_4)$ is a basis in  $\mathfrak{p}$ from Lemma \ref{base}, then conditions $F_{U}(0,s)=1/F(0,s)$ are equivalent
to equality $(e_1,e_2)=0.$ 
\end{remark}

\section{Abnormal extremals in the case $\dim(\mathfrak{p})=3$}

\begin{proposition}
\label{threed}
A four-dimensional connected Lie group $G$ with Lie algebra $\mathfrak{g}$ and a three-dimensional generating subspace  $\mathfrak{p}\subset\mathfrak{g}$ has abnormal extremals $($for arbitrary left-invariant quasimetric $d$ on $G,$
defined by a seminorm $F$ on $\mathfrak{p})$ if and only if $\mathfrak{p}_1=\mathfrak{p}\cap\mathfrak{N}(\mathfrak{p})\neq\{0\},$ where $\mathfrak{N}(\mathfrak{p})$ is the normalizer of $\mathfrak{p}$ in $\mathfrak{g}.$ Furthermore, $\dim(\mathfrak{p}_1)=1$ and every one-parameter subgroup $g=g(t)=\exp(tX),$ where $X\in\mathfrak{p}_1,$ $F(X)=1,$ is an abnormal extremal for $(G,d);$ there is no other abnormal extremal with origin $e\in G$. In addition, the extremal $g$ is strictly abnormal $($non-strictly abnormal$)$
for every quasimetric $d$ if and only if $\mathfrak{p}_1\subset [\mathfrak{p}_1,\mathfrak{p}]$ $($respectively, 
$\mathfrak{p}_1=\mathfrak{p}\cap\mathfrak{C}(\mathfrak{p}),$ where $\mathfrak{C}(\mathfrak{p})$ is the centralizer of
$\mathfrak{p}$ in $\mathfrak{g})$. 
\end{proposition} 

\begin{proof}
Let we have an abnormal extremal $g(t),$ $t\in I,$  where $I\subset\mathbb{R}$ is some nonempty open connected subset. 
It is defined by conditions (\ref{difur}) for $u(t)\in U,$ $0\neq \psi(t)\in\mathfrak{g}^{\ast},$ $t\in I,$ where $u(t)$ and
$\psi(t)$ are respectively measureable and absolutely continuous functions such that $\psi(t)|\mathfrak{p}\equiv 0,$
$t\in I.$ Consequently, for almost all $t\in I,$ $\ad u(t)(\mathfrak{p})\subset\mathfrak{p},$ i.e. $u(t)\in\mathfrak{p}_1$ from
the statement of Proposition \ref{threed}. Since $\dim(\mathfrak{p})=3,$ $\mathfrak{p}$ generates $\mathfrak{g}$, and $F(u(t))=1,$ 
then it is clear that $\dim(\mathfrak{p}_1)=1,$ $u(t)$ is defined uniquely for these $t$ and we may suppose that 
$u(t)\equiv X$ for some $0\neq X\in \mathfrak{p}_1$. 

Assume now that $\mathfrak{p}_1\neq \{0\},$ hence $\dim(\mathfrak{p}_1)=1$. Let $X\in\mathfrak{p}_1,$ $F(X)=1$ and 
$\psi_0|\mathfrak{p}\equiv 0$ for some non-zero covecor $\psi_0\in\mathfrak{g}^{\ast}.$ It is evident that under supposition
$u(t)\equiv X$ there exist unique solutions $g(t)=\exp(tX),$ $\psi(t),$ $t\in\mathbb{R}$, of the system (\ref{difur}) 
with the initial conditions $g(0)=e,$ $\psi(0)=\psi_0.$ In addition, $\psi(t)|\mathfrak{p}\equiv 0,$ $t\in\mathbb{R},$ because  $X\in\mathfrak{p}_1;$ moreover, $\psi(t)\equiv \psi_0$ if $\ad X(\mathfrak{g})\subset \mathfrak{p}.$
Consequently, $g(t),$ $t\in\mathbb{R},$ is an abnormal extremal with respect to covector function $\psi(t),$ $t\in\mathbb{R}.$

Let us suppose that the above extremal $g(t)$ is non-stricly abnormal for $(G,d)$, i.e. for some function $\psi(t)\in\mathfrak{g}^{\ast},$ $1\equiv \psi(t)(X)=\max_{u\in U}\psi(t)(u)$ and (\ref{difur}) is fulfilled. 
Then $\psi(t)(v)=0$ for all $t\in\mathbb{R}$ and all $v\in [\mathfrak{p}_1,\mathfrak{p}].$ This maybe only in two cases: 
1) $[\mathfrak{p}_1,\mathfrak{p}]=\{0\},$ i.e. $\mathfrak{p}_1=\mathfrak{p}\cap\mathfrak{C}(\mathfrak{p});$ 
2) $[\mathfrak{p}_1,\mathfrak{p}]\neq \{0\},$ but this set is parallel in $\mathfrak{p}$ to the set 
$W=\{w\in \mathfrak{p}:\psi(0)(w)=1\},$ which is one of supporting plane to $U$ at its boundary point $X,$ hence 
$X\notin \ad X(\mathfrak{p})$. It is clear from here that $g(t),$ $t\in\mathbb{R},$ is a non-strictly abnormal extremal
for all $d$ in the case 1), and only for some $d$ if $\mathfrak{p}_1\not\subset [\mathfrak{p}_1,\mathfrak{p}]\neq \{0\};$ 
and $g$ is strictly abnormal extremal for all $d$ if and only if $\mathfrak{p}_1\subset [\mathfrak{p}_1,\mathfrak{p}].$ 
\end{proof}

\begin{remark}
Proposition \ref{threed} is valid for a connected Lie group $G$ of dimension $n\geq 5$ and a subspace  $\mathfrak{p}\subset\mathfrak{g}$ of codimension 1 with possible violation of equality $d=\dim(\mathfrak{p}_1)=1$ if $\mathfrak{p}_1\neq \{0\}$ (one can guarantee only inequalities $1\leq d\leq n-3$). 	
\end{remark}

\section{Examples}

The next theorem is a summary of Propositions \ref{solv3}, \ref{solv41}, \ref{solv410}, \ref{prop1} and Theorems \ref{simple}, \ref{solv49}, \ref{main1}.

\begin{theorem}
\label{mmm}
Let $G$ be a four-dimensional connected Lie group with Lie algebra $(\mathfrak{g},[\cdot,\cdot])$, 
$\mathfrak{p}\subset\mathfrak{g}$ be a two-dimensional subspace, generating $\mathfrak{g}$ by operation $[\cdot,\cdot]$, 
$d$ be arbitrary left-invariant quasimetric on $G$ defined by some seminorm $F$ on $\mathfrak{p}$. Then
	
1. Every abnormal extremal in $(G,d)$ is non-strictly abnormal in any of the following cases:   
1)  $\mathfrak{g}=\mathfrak{g}^0_{4,8}$; 2) $\mathfrak{g}=\mathfrak{g}_{4,10}$; 3) $\mathfrak{g}$ is decomposable and has
a solvable indecomposable three-dimensional subalgebra; 4) $\mathfrak{g}$ is indecomposable and has a three-dimensional
commutative ideal.

2. Let $\mathfrak{g}$ be one of Lie algebras $\mathfrak{g}_{3,7}\oplus\mathfrak{g}_1$, $\mathfrak{g}_{4,7}$, $\mathfrak{g}^{\alpha}_{4,8}$, $-1\leq\alpha<1$, $\alpha\neq 0$, $\mathfrak{g}^{\alpha}_{4,9}$, $\alpha\geq 0$,  $\mathfrak{g}=\mathfrak{g}_{3,6}\oplus\mathfrak{g}_1$ and the restriction of the Killing form $k_{\mathfrak{g_{3,6}}}$ 
to the projection of the subspace $\mathfrak{p}$ onto $\mathfrak{g}_{3,6}$ is negative definite. Then abnormal
extremal (\ref{dif}) (and every its left shift) in $(G,d)$ is non-strictly abnormal if and only if $F_U(0,s)=1/F(0,s),$ 
$s=\pm 1$.
	
3. If $\mathfrak{g}=\mathfrak{g}_{3,6}\oplus\mathfrak{g}_1$ and the restriction of the Killing form $k_{\mathfrak{g_{3,6}}}$ 
to the projection of $\mathfrak{p}$ onto $\mathfrak{g}_{3,6}$ is non-degenerate and with alternating signs, then every
abnormal extremal in $(G,d)$ is strictly abnormal.  

\end{theorem}

\begin{example}
Every abnormal extremal on the Engel group with nilpotent Lie algebra $\mathfrak{g}_{4,1}=\mathfrak{eng}$ and any
left-invariant quasimetric $d$, defined by some seminorm $F$ on $\mathfrak{p}=\rm{span}(E_1,E_3,E_4)\subset\mathfrak{eng},$
is one of two 1-parameter subgroups $g(t)=\exp(tsE_1/F(sE_1)),$ $s=\pm 1,$ or its left shift. In consequence of
Proposition \ref{threed} and centrality of $E_1$ these extremals are non-strictly abnormal.	
\end{example}

\begin{example}
\label{g43}	
Every abnormal extremal on connected (solvable) Lie group $G$ with Lie algebra $\mathfrak{g}=\mathfrak{g}_{4,3}$ and any
left-invariant quasimetric $d$, defined by some seminorm $F$ on $\mathfrak{p}=\rm{span}(E_1,E_3,E_4)\subset\mathfrak{g},$
is one of two 1-parameter subgroups $g(t)=\exp(tsE_1/F(sE_1)),$ $s=\pm 1,$ or its left shift. In consequence of
Proposition \ref{threed} and equality $[E_1,E_4]=E_1$ these extremals are strictly abnormal.	
\end{example}

An interesting partial case of the last example is $(G,d_1)$ with left-invariant sub-Riemannian metric $d_1$ 
defined by the inner product $(\cdot,\cdot)$ on $\mathfrak{p}$ with the orthonormal basis $(E_1,E_3,E_4).$

\begin{proposition}
\label{nong}	
Strictly abnormal extremals $g(t)=\exp(stE_1),$ $s=\pm 1,$ in $(G,d_1)$ are not geodesics, i.e. none their 
segment $g(t),$ $t\in [t_0,t_1],$ where $t_0< t_1,$ is shortest.	
\end{proposition}	

\begin{proof} Proposition is true because of the equality $\mathfrak{p}+[\mathfrak{p},\mathfrak{p}]=\mathfrak{g}$ and 
so-called Goh optimality condition for abnormal geodesics (see \cite{AS}) which implies that every parameterized by
arc length geodesic in $(G,d_1)$ is normal. We shall give an independent geometric proof in this case, 
reducing it to the question on geodesics of left-invariant Riemannian metric on a simply connected non-compact non-commutative two-dimensional Lie group.

Subgroup $N=\{\exp(tE_2), t\in\mathbb{R}\}\subset G$ is a central subgroup of the isometry group of $(G,d_1),$ 
consisting of left shifts $l_g,$ $g\in G.$ Therefore the canonical projection $p_3: (G,d_1)\rightarrow (G_3=G/N,d_3)$
onto $G_3$ with left-invariant Riemannian metric $d_3,$ defined by the inner product $(\cdot,\cdot)$ with the orthonormal
basis $(E_1,E_3,E_4)$ of Lie algebra $\mathfrak{g}_3$ of this group with relation $[E_1,E_4]=E_1$ and central element $E_3,$
is a submetry \cite{BerGui}. Notice that $G_3\cong G_2\times H,$ where $G_2$ is a two-dimensional non-commutative
Lie group, while $H=\{\exp(tE_3),t\in\mathbb{R}\}$ is a central subgroup of the group $G_3.$ Therefore the canonical
projection $p_2: (G_3,d_3)\rightarrow (G_2=G_3/H,d_2)$ onto $G_2$ with left-invariant Riemannian metric $d_2,$ 
defined by the inner product $(\cdot,\cdot)$ with the orthonormal basis $(E_1,E_4)$ of Lie algebra $\mathfrak{g}_2$ 
of this Lie group with relation $[E_1,E_4]=E_1$, is a submetry, moreover, a Riemannian submersion \cite{BerGui}. 
Consequently, the composition $p=(p_2\circ p_3): (G,d_1)\rightarrow (G_2,d_2)$ is a submetry and so 
$g_2(t)=p(g(t))=\exp(tE_1)\in G_2,$ $t\in\mathbb{R},$ is geodesic if $g(t),$ $t\in\mathbb{R},$ is geodesic. 
But this is impossible because $(G_2,d_2)$ is isometric to Lobachevsky plane with constant sectional curvature -1, 
and in the well-known conformal model of this plane in the upper semiplane, $g_2(t),$ $t\in\mathbb{R},$ is a horizontal 
straight line which is not geodesic \cite{Mil}, \cite{BerZub}. Thus obtained controversy proves Proposition \ref{nong}. 

\end{proof} 

\begin{example}
(\cite{LZ}, Section 9.5.)
Let us consider Lie group $G=\SO(3)\times \mathbb{R}$ with Lie algebra $\mathfrak{g}_{3,7}\oplus\mathfrak{g}_1$  
and left-invariant sub-Riemannian metric $d,$ defined by the inner product $(\cdot,\cdot)$ on subspace
$\mathfrak{p}=\rm{span}(f,g)$ with the orthonormal basis $(f,g),$ where $f=E_1+E_4,$ $g=E_1+E_2+2E_4.$ 
It is stated that one-parameter subgroups $g_s(t)=\exp(stg),$ $t\in\mathbb{R},$ $s=\pm 1,$ are unique 
(strictly) abnormal parameterized by arc length geodesics with origin $e.$ 
\end{example}
	
Let us compare this with our results. It is clear that $e_1=f,$ $e_2=g,$ 
$$e_3=[e_1,e_2]=E_3,\quad e_4=[e_1,e_3]=-E_2,\quad [e_2,e_3]=E_1-E_2= 2e_1 -e_2.$$
Thus, $C_{23}^4=0,$ $C_{23}^1=2\neq 0,$ but $C_{23}^2=-1\neq 0.$ In consequence of Proposition \ref{prop1}, $g_s,$ $s=\pm 1$, 
are unique abnormal extremals with origin $e.$ Following to Lemma \ref{base}, in order that $C^2_{23}$ to become zero, 
we change $e_1$ by $$e^{\prime}_1 =e_1 + (C^2_{23}/C^1_{23})e_2 = e_1-e_2/2 = E_1+E_4-(E_1+E_2+2E_4)/2=(E_1-E_2)/2.$$ 
Since $(e^{\prime}_1,e_2)\neq 0,$ then it follows from p.~2 of Theorem \ref{mmm} and Remark \ref{subr} that abnormal extremals $g_s,$ 
$s=\pm 1,$ are strictly abnormal.

\end{document}